\pdfoutput=1
\RequirePackage{ifpdf}
\ifpdf 
\documentclass[pdftex]{sigma}
\else
\documentclass{sigma}
\fi

\usepackage{cancel}

\numberwithin{equation}{section}

\newtheorem{Theorem}{Theorem}[section]
\newtheorem{Corollary}[Theorem]{Corollary}
\newtheorem{Proposition}[Theorem]{Proposition}
\newtheorem{Lemma}[Theorem]{Lemma}
 { \theoremstyle{definition}
\newtheorem{Definition}[Theorem]{Definition}
\newtheorem{Remark}[Theorem]{Remark} }

\def\dd{{\rm d}}

\begin{document}

\allowdisplaybreaks

\newcommand{\arXivNumber}{1804.05688}

\renewcommand{\thefootnote}{}

\renewcommand{\PaperNumber}{087}

\FirstPageHeading

\ShortArticleName{Notes on Non-Generic Isomonodromy Deformations}

\ArticleName{Notes on Non-Generic Isomonodromy Deformations\footnote{This paper is a~contribution to the Special Issue on Painlev\'e Equations and Applications in Memory of Andrei Kapaev. The full collection is available at \href{https://www.emis.de/journals/SIGMA/Kapaev.html}{https://www.emis.de/journals/SIGMA/Kapaev.html}}}

\Author{Davide GUZZETTI}
\AuthorNameForHeading{D.~Guzzetti}

\Address{SISSA, Via Bonomea 265 - 34136 Trieste, Italy}
\Email{\href{mailto:guzzetti@sissa.it}{guzzetti@sissa.it}}

\ArticleDates{Received April 17, 2018, in final form August 14, 2018; Published online August 21, 2018}

\Abstract{Some of the main results of [Cotti G., Dubrovin B., Guzzetti D., \textit{Duke Math.~J.}, to appear, arXiv:1706.04808], concerning non-generic isomonodromy deformations of a certain linear differential system with irregular singularity and coalescing eigenvalues, are reviewed from the point of view of Pfaffian systems, making a~distinction between weak and strong isomonodromic deformations. Such distinction has a~counterpart in the case of Fuchsian systems, which is well known as Schlesinger and non-Schlesinger deformations, reviewed in Appendix~A.}


\Classification{34M56; 34M35; 34M40}

\renewcommand{\thefootnote}{\arabic{footnote}}
\setcounter{footnote}{0}

\section{Introduction}
These notes partly touch the topics of my talk in Ann Arbor at the conference in memory of Andrei Kapaev, August 2017. They are a reworking of some of the main results of \cite{CDG}, concerning non-generic isomonodromy deformations of the differential system \eqref{11marzo2018-1} below. The approach here is different from \cite{CDG}, since I will start from the point of view of Pfaffian systems. This allows to introduce the main theorem (Theorem \ref{25marzo2018-1} below) in a relatively simple way (provided we give for granted another result of \cite{CDG} summarised in Theorem \ref{23marzo2018-3} below). The approach here is also an opportunity to review the difference between {\it weak} and {\it strong} isomonodromy deformations.

In Section \ref{19marzo2018-1}, we define ``weak'' isomonodromic deformations of the system~\eqref{11marzo2018-1}, through the Pfaffian system~\eqref{11marzo2018-2} below, that we characterise together with its fundamental matrix solutions. The residue matrix~$A(u)$ is not assumed to be diagonalizable or non-resonant, so that a first non-generic issue is included in the discussion. In Section~\ref{19marzo2018-2}, we define and characterise ``strong'' isomonodromy deformations of the differential system~\eqref{11marzo2018-1} in terms of solutions of the Pfaffian system and in terms of essential monodromy data. We recover the total differential system used in the last part of~\cite{CDG}, a special case of which (for $A(u)$ skew-symmetric) has been well known in the theory of Frobenius manifolds \cite{Dub1,Dub2}. In Section~\ref{25marzo2018-3}, we explain some of the main results of \cite{CDG}, particularly Theorem~\ref{25marzo2018-1}, which extend strong isomonodromy deformations to the non-generic case when the matrix $\Lambda= \operatorname{diag}(u_1,\dots,u_n)$ in system~\eqref{11marzo2018-1} has coalescing eigenvalues $(u_i-u_j)\to 0$ for some $1\leq i \neq j \leq n$. The proof of Theorem \ref{25marzo2018-1} is based on Theorem \ref{23marzo2018-3}, which holds for a differential system more general than~\eqref{11marzo2018-1}, not necessarily isomonodromic (see system \eqref{22marzo2018-6}). Paper~\cite{CDG} is unavoidably long because of a careful set up for the background and the proof of Theorem~\ref{23marzo2018-3}. Given it for granted, in these notes we can introduce and prove Theorem~\ref{25marzo2018-1} in a relatively short manner, starting from the discussion of the Pfaffian system~\eqref{11marzo2018-2}.

At the end of these notes, I would like to show how the notion of weak and strong isomo\-nodromy deformations also arises naturally in the framework of Fuchsian systems; in this case, ``weak'' and ``strong'' deformations respectively coincide with {\it non-Schlesinger} and {\it Schlesinger} deformations. We will review this issue in Appendix~\ref{appendixA}.

Several detailed comments on the existing literature concerning non-generic isomonodromy deformations have been included in the introduction of~\cite{CDG}, so we do not repeat them here. However, a reference is missing from the cited introduction (in the arXiv version), which I will include here in Remark~\ref {10aprile2018-1} of Section~\ref{25marzo2018-3} below.

\section{A Pfaffian system defining weak isomonodromic deformations}\label{19marzo2018-1}

In this section, particularly in Proposition \ref{16marzo2018-1}, we characterise a Pfaffian system responsible for the (weak) isomonodromic deformations of the differential system
\begin{gather}\label{11marzo2018-1}
\partial_z Y=\left(\Lambda+\frac{A(u)}{z}\right)Y.
\end{gather}
The above appears naturally as the (inverse) Laplace transform of a Fuchsian system of Okubo type, with poles at $u_1,\dots,u_n$ \cite{BJL4,Dub4,guz2016,SCHA}. A particular case of \eqref{11marzo2018-1} is at the core of the isomonodromic deformation approach to Frobenius manifolds \cite{Dub1,Dub2}.

Consider an $n\times n$ matrix Pfaffian system
\begin{gather}\label{11marzo2018-2}
\dd Y=\omega Y,\qquad \omega=\omega_0(z,u)\dd z+\sum_{j=1}^n \omega_j(z,u) \dd u_j.
\end{gather}
The complex $n+1$ variables will be denoted by $(z,u)$, with $u:=(u_1,\dots,u_n)$. We assume that
\begin{gather}\label{11marzo2018-6}
\omega_0=\Lambda+\frac{A(u)}{z},
\end{gather}
with $\Lambda=\operatorname{diag}(u_1,\dots,u_n)$ and $A(u)$ an $n\times n$ matrix, so that \eqref{11marzo2018-2} can be viewed as a deformation of the differential system~\eqref{11marzo2018-1}. We suppose that $\omega_1,\dots,\omega_n$ are holomorphic of \smash{$(z,u)\in \mathbb{C}\times \mathbb{D}\big(u^0\big)$}, where $\mathbb{D}\big(u^0\big)$ is a polydisc centred at $u^0=\big(u_1^0,\dots,u_n^0\big)$, contained in
\begin{gather*}
\mathbb{C}^n\backslash \bigcup_{i\neq j} \{u_i-u_j=0\}.
\end{gather*}
 We also assume that $A(u)$ is holomorphic on $\mathbb{D}\big(u^0\big)$ and that $z=\infty$ is at most a pole of the~$\omega_j(z,u)$, so that \eqref{11marzo2018-2} is meromorphic on $\mathbb{P}^1\times \mathbb{D}\big(u^0\big)$. The complete Frobenius integrability of the system \eqref{11marzo2018-2} is expressed by
\begin{gather*}
\dd\omega =\omega \wedge \omega,
\end{gather*}
namely, letting $(x_0,x_1,\dots,x_n):=(z,u_1,\dots,u_n)$,
\begin{gather}\label{11marzo2018-3}
\frac{\partial \omega_\beta}{\partial x_\alpha}+\omega_\beta\omega_\alpha=\frac{\partial \omega_\alpha}{\partial x_\beta}+\omega_\alpha\omega_\beta, \qquad \alpha\neq\beta=0,1,\dots,n.
\end{gather}

If the integrability condition holds \cite{Bolibruch0,Haraoka,IKSY,YT}, then \eqref{11marzo2018-2} admits of a fundamental matrix solution $Y(z,u)$, which is holomorphic on $\mathcal{R}\times \mathbb{D}\big(u^0\big)$, where
\begin{gather*}\mathcal{R}:= \text{the universal covering of $\mathbb{P}^1\backslash\{0,\infty\}$}.
\end{gather*}
Its monodromy matrix $M$ associated with a simple loop $\gamma$ around $z=0$, defined by
\begin{gather*}Y(\gamma z,u)=Y(z,u)M, \end{gather*}
is independent of $u$ (and of course of $z$). The notation above means that $\gamma$ transforms $z\in \mathcal{R}$ to $\gamma z\in \mathcal{R}$ ($z$ and $\gamma z$ are in the same fibre over a point of $ \mathbb{P}^1\backslash \{0,\infty\}$, which we still denote by~$z$, to avoid heavy notations). To show that $M$ is constant, observe that $Y(\gamma z,u)$ is a solution (it is $Y(\cdot,u)$ seen as function on $\mathcal{R}$ evaluated at $\gamma z$). Therefore,
\begin{align}\label{25marzo2018-4-MENO}
\omega &=\dd Y(\gamma z,u) Y(\gamma z, u)^{-1}=\dd Y(z,u) Y(z,u)^{-1}+Y(z,u) \big(\dd M M^{-1}\big) Y(z,u)^{-1}\\
\label{25marzo2018-4}
&= \omega +Y(z,u) \big(\dd M M^{-1}\big) Y(z,u)^{-1} \quad \Longleftrightarrow \quad \dd M=0.
\end{align}
 In order to avoid heavy notations, we use the letter ``$\dd$'' either for the differential of a~func\-tion~$f(z,u)$ of variables~$(z,u)$ (like~$\dd Y$ above), and for the differential of a function of~$u$ alone (like~$\dd M$ above).

\begin{Definition}\label{18marzo2018-1}We call \eqref{11marzo2018-1}, regarded as the ``$z$-component'' of a Pfaffian system~\eqref{11marzo2018-2}, a~{\it weak} isomonodromic family of differential systems,
and~$Y(z,u)$ a~{\it weak} isomonodromic family of fundamental matrix solutions.
\end{Definition}
 By ``weak'' we mean that the monodromy matrix $M$ is constant, but other monodromy data, to be introduced below, may be not.

\begin{Lemma}\label{24marzo2018-2}Let $\omega_0$ be as in \eqref{11marzo2018-6}, let $\omega_1,\dots,\omega_n$ be holomorphic of $(z,u)\in \mathbb{C}\times \mathbb{D}\big(u^0\big)$, and~$A(u)$ holomorphic on $\mathbb{D}\big(u^0\big)$. The integrability condition \eqref{11marzo2018-3} implies that~$A$ satisfies
\begin{gather}\label{23marzo2018-8}
\frac{\partial A}{\partial u_j}=[\omega_j(0,u),A], \qquad j=1,\dots,n,
\end{gather}
and that the above is Frobenius integrable. Moreover, the system
\begin{gather}\label{23marzo2018-7}
\dd G=\left(\sum_{j=1}^n \omega_j(0,u)\dd u_j\right) G
\end{gather}
is Frobenius integrable.
\end{Lemma}

\begin{proof} Condition \eqref{11marzo2018-3} for $\beta=0$ and $\alpha=j\geq 1$ yields
\begin{gather}\label{24marzo2018-1}
\frac{\partial}{\partial u_j}
\left(\Lambda+\frac{A}{z}\right)+\left(\Lambda+\frac{A}{z}\right)\omega_j(z,u)=\frac{\partial \omega_j(z,u)}{\partial z}
+\omega_j(z,u)\left(\Lambda+\frac{A}{z}\right).
\end{gather}
Since $\omega_j(z,u)=\omega(0,u)+$ Taylor series in $z$, identification of the coefficients of $z^{-1}$ yields \eqref{23marzo2018-8}. The condition~\eqref{11marzo2018-3} for $\beta=i\geq 1$ and $\alpha=j\geq 1$ is regular at $z=0$. In particular, for $z=0$ we have
\begin{gather}\label{24marzo2018-3}
\frac{\partial \omega_i(0,u)}{\partial u_j}+\omega_i(0,u)\omega_j(0,u)=\frac{\partial \omega_j(0,u)}{\partial u_i}+\omega_j(0,u)\omega_i(0,u),\qquad i\neq j=1,\dots,n,
\end{gather}
which is the Frobenius integrability condition for \eqref{23marzo2018-7}. It is also the integrability condition of~\eqref{23marzo2018-8}. Indeed, let us compute (for brevity, we write $\omega_j(0):=\omega_j(0,u)$)
\begin{gather*}
\partial_k[\omega_j(0),A]-\partial_j[\omega_k(0),A]= \bigr(\partial_k\omega_j(0) A+\omega_j(0)\partial_kA-\partial_kA \omega_j(0)-A\partial_k\omega_j(0)\bigr)\\
\qquad\quad{} -\bigl(\partial_j\omega_k(0) A+\omega_k(0)\partial_jA-\partial_jA \omega_k(0)-A\partial_j\omega_k(0)\bigr)\\
\qquad{} =\bigr(\partial_k\omega_j(0) A+\omega_j(0)[\omega_k(0)A-A\omega_k(0)]-[\omega_k(0)A-A\omega_k(0)] \omega_j(0)-A\partial_k\omega_j(0)\bigr)\\
\qquad\quad{} -\bigr(\partial_j\omega_k(0) A+\omega_k(0)[\omega_j(0)A-A\omega_j(0)]-[\omega_j(0)A-A\omega_j(0)] \omega_k(0)-A\partial_j\omega_k(0)\bigr)\\
\qquad{} =\bigl\{[\partial_k\omega_j(0)+\omega_j(0)\omega_k(0)]-[\partial_j\omega_k(0)+\omega_k(0)\omega_j(0)]\bigr\}A\\
\qquad\quad{} + A\bigl\{[\omega_k(0)\omega_j(0)-\partial_k\omega_j(0)]-[\omega_j(0)\omega_k(0)-\partial_j\omega_k(0)]\bigr\}.
\end{gather*}
The last expression is zero for any possible $A$ if and only if~\eqref{24marzo2018-3} holds.
\end{proof}

The following is a standard result concerning the residue matrix at Fuchsian singularity of a~Pfaffian system (see for example~\cite{Haraoka}).

\begin{Proposition}\label{11marzo2018-7}Let $\omega_0$ be as in \eqref{11marzo2018-6}, let $\omega_1,\dots,\omega_n$ be holomorphic of $(z,u)\in \mathbb{C}\times \mathbb{D}\big(u^0\big)$, and $A(u)$ holomorphic on $\mathbb{D}\big(u^0\big)$. Then, $A(u)$ is holomorphically similar to a constant $($i.e., independent of $u)$ Jordan form $J$, namely there exists a holomorphic fundamental matrix $G(u)$ of \eqref{23marzo2018-7} on $\mathbb{D}\big(u^0\big)$, such that
\begin{gather*}
G(u)^{-1}A(u)G(u)=J.
\end{gather*}
In particular, this means that the weak isomonodromy deformation \eqref{11marzo2018-1} is isospectral.
\end{Proposition}

\begin{proof} Take a fundamental matrix $\tilde{G}(u)$ of \eqref{23marzo2018-7}, which is holomorphic on $\mathbb{D}\big(u^0\big)$. Then, using~\eqref{23marzo2018-8} and~\eqref{23marzo2018-7}, we have
\begin{gather*}
\partial_j\big(\tilde{G}^{-1}A\tilde{G}\big)=\tilde{G}^{-1} A \partial_j \tilde{G}+\tilde{G}^{-1}\partial_jA \tilde{G}-\tilde{G}^{-1}\partial_j\tilde{G} \tilde{G}^{-1}A \tilde{G}\\
\hphantom{\partial_j\big(\tilde{G}^{-1}A\tilde{G}\big)}{} =\tilde{G}^{-1}A \omega_j(0,u) \tilde{G} +\tilde{G}^{-1}\partial_jA \tilde{G} -\tilde{G}^{-1}\omega_j(0,u) A \tilde{G}\\
\hphantom{\partial_j\big(\tilde{G}^{-1}A\tilde{G}\big)}{}=\tilde{G}^{-1}\left([A,\omega_j(0,u)]+\partial_jA\right)\tilde{G}=0.
\end{gather*}
This implies that $\mathcal{A}:=\tilde{G}^{-1}A\tilde{G}$ is constant over $\mathbb{D}\big(u^0\big)$, and it has a constant Jordan form $J=\mathcal{G}^{-1}\mathcal{A}\mathcal{G}$ for some invertible matrix $\mathcal{G}$. Thus, the desired holomorphic matrix is $G(u)=\tilde{G}(u)\mathcal{G}$.
\end{proof}

At any fixed $u$, the differential system \eqref{11marzo2018-1} admits fundamental matrix solutions in {\it Levelt form} at $z=0$,
\begin{gather}\label{11marzo2018-8}
Y^{(0)}(z,u)=\widehat{Y}^{(0)}(z,u) z^D z^{L},
\end{gather}
where
\begin{gather}\label{3aprile2018-1}
\widehat{Y}^{(0)}(z,u)=G(u)\left(I+\sum_{j=1}^\infty \Psi_j(u)z^j\right)
\end{gather}
is a convergent matrix valued Taylor series ($I$ stands for the identity matrix). The matrix $G(u)$ puts $A(u)$ in Jordan form $J=G(u)^{-1}A(u)G(u)$ and satisfies~\eqref{23marzo2018-7} as in Proposition~\ref{11marzo2018-7}. The following characterisation holds \cite{AnosBol, Gant}. \begin{itemize}\itemsep=0pt
\item[--] $L$ is block-diagonal $L=L_1\oplus\cdots\oplus L_{\ell}$, with upper triangular matrices $L_q$; each $L_q$ has only one eigenvalue $\sigma_q$, satisfying $0\leq \operatorname{Re} \sigma_q <1$, and $\sigma_p\neq \sigma_q$ for $1\leq p\neq q\leq \ell$.
\item[--] $D$ is a diagonal matrix of integers, which can be split into blocks ${D}_1\oplus\cdots\oplus {D}_{\ell}$ as $L$. The integers ${d}_{q,r}$ in each ${D}_q=\operatorname{diag}({d}_{q,1},{d}_{q,2},\dots )$ form a non-increasing (finite) sequence ${d}_{q,1}\geq {d}_{q,2}\geq \cdots$.
\item[--] The eigenvalues of $A$ are $d_{q,s}+\sigma_q$, for $q=1,\dots,\ell$ and $s$ runs from 1 to some integer $n_q$, if the dimension of $L_q$ is $n_q\times n_q$ . Each block $L_q(u)$ corresponds to the eigenvalues of $A(u)$ which differ by non-zero integers.
\item[--] The expression in square brackets below is holomorphic at $z=0$ and the following limits hold
\begin{gather*}
 \lim_{z\to 0} \big[\widehat{Y}^{(0)}(z,u)\big(D+z^D L z^{-D}\big)\big(\widehat{Y}^{(0)}(z,u)\big)^{-1}\big]\nonumber\\
 \qquad{} = \lim_{z\to 0} \big[G(u)\big(D+z^D L z^{-D}\big)G(u)^{-1}\big]=A(u).
 \end{gather*}
Equivalently, $ D+ \lim\limits_{z\to 0} z^D L z^{-D}=J$.
\end{itemize}

A Levelt form \eqref{11marzo2018-8} is not uniquely determined by the differential system, since there is a~freedom in the choice of the coefficients~$\Psi_j(u)$ and the exponents~$D$ and~$L$. See~\cite{CDG} (and also~\cite{Dub2}) for more details.

The fundamental matrices \eqref{11marzo2018-8}, seen as solutions of the differential system~\eqref{11marzo2018-1} only, may have exponents~$L$ and $D$ depending on the point~$u$, and the coefficients $\Psi_j(u)$ are not guaranteed to be holomorphic. On the other hand, the Pfaffian system~\eqref{11marzo2018-2} is of Fuchsian type at $z=0$ with the only singular contribution in $\omega$ coming from $\omega_0(z,u)=A(u)/z+\Lambda$. For this reason, system \eqref{11marzo2018-2} also admits solutions of the form~\eqref{11marzo2018-8}, behaving nicely with respect to $u\in \mathbb{D}\big(u^0\big)$, according to the following

 \begin{Proposition} \label{30luglio 2018-3} Let $A(u)$ be holomorphic on $\mathbb{D}\big(u^0\big)$ and $\omega_1(z,u), \dots,\omega_n(z,u)$ be holomorphic on $ \mathbb{C}\times \mathbb{D}\big(u^0\big)$. Then:
 \begin{itemize}\itemsep=0pt
 \item the Pfaffian system \eqref{11marzo2018-2} admits fundamental matrix solutions in Levelt form \eqref{11marzo2018-8}--\eqref{3aprile2018-1};
 \item the exponents $D$ and $L$ are constant;
 \item the convergent series $\eqref{3aprile2018-1}$ defines a holomorphic matrix valued function on $\mathbb{C}\times \mathbb{D}\big(u^0\big)$; equivalently, $Y^{(0)}(z,u)$ is holomorphic on $\mathcal{R}\times \mathcal{U}_\epsilon\big(u^C\big)$.
 \end{itemize}
 \end{Proposition}
 Notice that Proposition~\ref{30luglio 2018-3} also holds in case $\mathbb{D}\big(u^0\big)$ intersects $\bigcup_{i\neq j }^n \{u_i-u_j=0\}$, provide that $A(u)$ and $\omega_1,\dots,\omega_n$ are holomorphic of $(z,u)\in \mathbb{C}\times \mathbb{D}\big(u^0\big)$ (see Corollary~\ref{31luglio2018-1}).

 \begin{proof} The Pfaffian system \eqref{11marzo2018-2} is of Fuchsian type at $z=0$. Thus, we can apply the result of \cite{Bolibruch0,YT}. In particular, by theorem 5 of \cite{YT}, there exists a gauge transformation
 \begin{gather} \label{31luglio2018-3}
 Y=V(z,u) \widetilde{Y},
\end{gather}
 holomorphic in $\mathbb{C}\times \mathbb{D}\big(u^0\big)$ and represented by a convergent matrix valued series
 \begin{gather*}
 V(z,u)=I+\sum_{k\colon k_1+\cdots +k_n>0} V_k z^{k_0}\big(u-u_1^0\big)^{k_1}\cdots \big(u-u_n^0\big)^{k_n}
 \end{gather*}
 (here $k=(k_0,k_1,\dots,k_n)$ and $V_k$ are matrices), such that
 \begin{gather} \label{31luglio2018-4}
 \dd\widetilde{Y}=\big[ \omega_0\big(z,u^0\big)\dd z\big] \widetilde{Y} \equiv \left[ \left(\frac{A\big(u^0\big)}{z}+\Lambda\big(u^0\big)\right)\dd z\right] \widetilde{Y}.
 \end{gather}
 The above admits fundamental solutions in Levelt form
 \begin{gather} \label{31luglio2018-5}
 \widetilde{Y}^{(0)}(z)=W(z) z^D z^L,\qquad W(z)=W_0 +\sum_{\ell=1}^\infty W_\ell z^\ell, \qquad \det(W_0)\neq 0,
 \end{gather}
which are independent of $u$ ($u^0$ being fixed). Here, $W(z)$ converges and defines a holomorphic function on $\mathbb{C}$. This means that
 \begin{gather*}
 Y^{(0)}(z,u):=V(z,u)W(z) \widetilde{Y}^{(0)}(z)
 \end{gather*}
is a fundamental matrix solution of~\eqref{11marzo2018-2} in Levelt form, with constant exponents~$L$ and~$D$. Moreover, $V(z,u)W(z) $ is holomorphic on $\mathbb{C}\times \mathbb{D}\big(u^0\big)$ and its series representation can be rearranged as follows
 \begin{gather*}
 G(u) \left(I+\sum_{j=1}^\infty \Psi_j(u)z^j\right), \qquad G(u)= V(0,u)W_0.
 \end{gather*}
 Notice that $G(u)$ puts $A(u)$ in Jordan form. The proof above can be also deduced by directly applying Theorems~4 and~7 of~\cite{YT}.
 \end{proof}

Chosen a holomorphic isomonodromic family $Y(z,u)$ of \eqref{11marzo2018-2}, with holomorphic $\omega_1,\dots,\omega_n$ at $z=0$, then there exists a connection matrix $H$ such that
\begin{gather*} 
 Y(z,u)=Y^{(0)}(z,u)H.
 \end{gather*}
By Proposition \ref{30luglio 2018-3}, also $Y^{(0)}$ is an isomonodromic family of fundamental solutions, so that $H$ cannot depend on $u$. Indeed
\begin{align*}
\omega =\dd Y Y^{-1}&= \dd Y^{(0)} \big(Y^{(0)}\big)^{-1}+ Y^{(0)} \dd H H^{-1} \big(Y^{(0)}\big)^{-1}
\\
&= \omega +Y^{(0)} \dd H H^{-1} \big(Y^{(0)}\big)^{-1}\quad \Longleftrightarrow \quad \dd H=0.
\end{align*}
Without loss of generality, we can always choose \begin{gather*}Y(z,u)=Y^{(0)}(z,u).\end{gather*}

\begin{Remark} \label{18marzo2018-2} For the sake of computations, we can write $L=\Sigma+N$, where $\Sigma=\Sigma_1\oplus \cdots \oplus \Sigma_\ell$ is diagonal, with $\Sigma_q=\sigma_q I_{\dim (\Sigma_q)}$ (here $I_{\dim (\Sigma_q)}$ stands for the identity matrix of dimension $\dim (\Sigma_q)$), while $N=N_1\oplus \cdots \oplus N_\ell$ is nilpotent. Since~$\Sigma_q$ and~$N_q$ have the same dimension ($q=1,\dots,\ell$), we have
 \begin{gather*}
 [\Sigma,N]=0,
 \end{gather*}
 and thus
 \begin{gather*}
 z^D z^L=z^D z^\Sigma z^N = z^{D+\Sigma} z^N =z^{D+\Sigma}\sum_{\ell=0}^{\overline{k}} \frac{N^\ell}{\ell!} (\log z)^\ell,
 \end{gather*}
 for some finite $\overline{k}$.
 \end{Remark}

The differential system \eqref{11marzo2018-1} also admits a family of fundamental matrices $Y_r(z,u)$, $r\in \mathbb{Z}$, uniquely determined by their asymptotic behaviour in suitable sectors of central angular opening greater that~$\pi$. We fix a half-line of direction $\arg z=\tau$ in $\mathcal{R}$ which does not coincide with any of the Stokes rays associated with $\Lambda\big(u^0\big)$, namely with the half lines in~$\mathcal{R}$ specified by
\begin{gather*}
 \operatorname{Re} \big(z\big(u_i^0-u_j^0\big)\big)=0,\qquad \operatorname{Im} \big(z\big(u_i^0-u_j^0\big)\big)<0.
\end{gather*}
Notice that in order to obtain half-lines, we need to specify the sign the imaginary part. We say that $\tau$ is an {\it admissible direction at $u^0$}. Notice that we are working in the universal covering~$\mathcal{R}$, so that any $\tau+h \pi$, $h\in \mathbb{Z}$ is also an admissible direction. We define the Stokes rays associated with~$\Lambda(u)$ to be the infinitely many half-lines in $\mathcal{R}$ defined by
\begin{gather*}
\operatorname{Re}(z(u_i-u_j))=0,\qquad \operatorname{Im}(z(u_i-u_j))<0.\end{gather*}
These rays rotate as $u$ varies in $\mathbb{D}\big(u^0\big)$, and may cross an admissible direction. Let $\mathbb{D}^\prime\subset \mathbb{D}\big(u^0\big)$ be a disk sufficiently small so that no Stokes rays cross a direction $\tau+h \pi$ as $u$ varies in the closure of $\mathbb{D}^\prime$. Then, for $u\in \overline{\mathbb{D}^\prime}$, we construct a sector $\mathcal{S}_{1}(u)$, by considering the ``half plane'' $\tau-\pi<\arg z <\tau$ and extending it to the nearest Stokes rays associated with $\Lambda(u)$ lying outside the half plane. Analogously, for any $r\in \mathbb{Z}$ a sector $\mathcal{S}_{r}(u)$ is obtained extending the ``half plane'' $\tau+(r-2)\pi<\arg z <\tau+(r-1)\pi$ to the nearest Stokes rays outside of it. The sectors $\mathcal{S}_r(u)$ have central angular opening greater than $\pi$, and the same holds for
\begin{gather*}
 \mathcal{S}_r(\overline{ \mathbb{D}^\prime}):=\bigcap_{u\in\overline{ \mathbb{D}^\prime} } \mathcal{S}_r(u).
 \end{gather*}

Since $u_i\neq u_j$ for $i\neq j$, it is well known that the system \eqref{11marzo2018-1} has a {\it formal fundamental matrix solution}
\begin{gather*}
Y_F(z,u)=F(z,u) z^{B(u)}e^{z\Lambda},
\end{gather*}
where \begin{gather*}B(u):=\operatorname{diag}\big(A_{11}(u),\dots,A_{nn}(u)\big)\end{gather*} is a diagonal matrix, and $F(z,u)$ is the formal matrix-valued expansion
\begin{gather}\label{14marzo2018-1}
F(z,u)=I+\sum_{k=1}^\infty F_k(u)z^{-k},
\end{gather}
with holomorphic on $\mathbb{D}\big(u^0\big)$ coefficients $F_k(u)$ uniquely determined by \eqref{11marzo2018-1}. Moreover, for each fixed $u$, and for any sector $S$ containing in its interior {\it a set of basic Stokes rays}\footnote{Let $\mu$ be an integer. We say that a finite sequence of Stokes rays
\begin{gather*}
\arg z=\tau_1,\quad \arg z=\tau_2, \quad \dots ,\quad \arg z=\tau_\mu, \quad\tau_1<\tau_2<\cdots <\tau_\mu,\quad (\tau_\mu-\tau_1)<\pi,
\end{gather*}
form a {\it set of basic Stokes rays} if all the other Stokes rays can be obtained as
\begin{gather*}
\arg z= \tau_\nu +k\pi, \qquad \text{for some $\nu\in \{1,2,\dots,\mu\}$ and $k\in\mathbb{Z}$}.
\end{gather*}
There exist a set of basic rays (and then infinitely many). Indeed, a sector with angular opening $\pi$, whose boundary rays are not Stokes rays, contains exactly a set of basic rays.} and no other Stokes rays, it is well known \cite{BJL1} that there exists a unique fundamental matrix solution
\begin{gather*}
Y_{(S)}(z,u)=\widehat{Y}_{(S)}(z,u) z^{B(u)}e^{z\Lambda}
\end{gather*}
such that
\begin{gather*}
\widehat{Y}_{(S)}(z,u)\sim F(z,u),\qquad z\to \infty \ \ \text{in} \ \ S.
\end{gather*}
Now, if $u$ is restricted to $\mathbb{D}^\prime$, then we can take a family of such solutions, labelled by $r\in \mathbb{Z}$, with $S=\mathcal{S}_r(u)$,
\begin{gather}\label{11marzo2018-10-BIS}
Y_r(z,u)=\widehat{Y}_r(z,u) z^{B(u)}e^{z\Lambda},\\
\label{11marzo2018-10} \widehat{Y}_r(z,u)\sim F(z,u),\qquad z\to \infty \ \ \text{in} \ \ \mathcal{S}_r(u).
\end{gather}
It is a fundamental result of Sibuya's \cite{Sh11,Sh4,Sh2} that each $Y_r(z,u)$ depends holomorphically on $u\in \mathbb{D}^\prime$, and the asymptotics \eqref{11marzo2018-10} is uniform for $z\to \infty$ in $\mathcal{S}_r(\overline{ \mathbb{D}^\prime})$ with respect to~$u$ varying in~$\mathbb{D}^\prime$. We notice that it may be necessary to further restrict~$\mathbb{D}^\prime$, because Sibuya's result requires~$\mathbb{D}^\prime$ to be sufficiently small in order for the holomorphic dependence to occur. The {\it Stokes matrices}~$\mathbb{S}_r$ defined by
 \begin{gather*}
 Y_{r+1}(z,u)=Y_r(z,u)\mathbb{S}_r(u),
 \end{gather*}
 are holomorphic on $\mathbb{D}^\prime$. We recall that the structure of the Stokes matrices is such that \smash{$\operatorname{diag}( \mathbb{S}_r)=I$}, and
 \begin{gather*}
 (\mathbb{S}_r)_{ij}=0 \quad \text{for $i$, $j$ such that } e^{(u_i-u_j)z}\to \infty, \qquad z\in \mathcal{S}_r(\overline{ \mathbb{D}^\prime}) \cap \mathcal{S}_{r+1}(\overline{ \mathbb{D}^\prime}).
 \end{gather*}
 This is a ``triangular structure''. Successive matrices $\mathbb{S}_r$ and $\mathbb{S}_{r+1}$ have opposite ``triangular'' structures.

Given the holomorphic isomonodromic family $Y(z,u)$, there will exist holomorphic connection matrices $H_r(u)$, $u\in\mathbb{D}^\prime$, such that
\begin{gather}\label{14marzo2018-4}
 Y(z,u)=Y_r(z,u) H_r(u), \qquad u\in\mathbb{D}^\prime.
 \end{gather}

Let $E_j$ be the matrix with the only non-zero entry being $(E_j)_{jj}=1$. Notice that for distinct eigenvalues we have $\partial_j\Lambda=E_j$. The following proposition holds.

\begin{Proposition}\label{16marzo2018-1}Consider a completely integrable linear Pfaffian system
\begin{gather}\label{29marzo2018-1}
\dd Y=\omega Y,\qquad \omega=\left(\Lambda+\frac{A(u)}{z}\right)\dd z+\sum_{j=1}^n \omega_j(z,u) \dd u_j,
\end{gather}
with $A(u)$ holomorphic on $\mathbb{D}\big(u^0\big)$ and $\omega_1(z,u), \dots,\omega_n(z,u)$ holomorphic on $ \mathbb{C}\times \mathbb{D}\big(u^0\big)$. Let $z=\infty$ be at most a pole of $\omega_1,\dots,\omega_n$. Then:
\begin{itemize}\itemsep=0pt
\item[$(A)$] each $\omega_j(z,u)$ is linear in $z$ and determined by $A(u)$ up to an arbitrary holomorphic diagonal matrix $\mathcal{D}_j(u)$ on $\mathbb{D}\big(u^0\big)$, with the following structure\footnote{Explicitly, \begin{gather*}\left(
 \frac{A_{ab}(\delta_{aj}-\delta_{bj})}{u_a-u_b}
 \right)_{a,b=1}^n
 =\left(
\begin{matrix}
0 & 0&\dfrac{ -A_{1j}}{u_1-u_j} & 0& 0
 \\
0 &0 &\vdots & 0& 0
 \\
 \dfrac{ A_{j1}}{u_j-u_1} & \cdots&0 & \cdots&\dfrac{ A_{jn}}{u_j-u_n}
 \\
0 &0 &\vdots &0 & 0
 \\
 0 &0 &\dfrac{ -A_{nj}}{u_n-u_j} & 0& 0
\end{matrix}
\right)
\end{gather*}
}
\begin{gather*}
\omega_j(z,u)= zE_j+\omega_j(0,u),\\
\omega_j(0,u):=\left( \frac{A_{ab}(u) (\delta_{aj}-\delta_{bj})}{u_a-u_b} \right)_{a,b=1}^n+ \mathcal{D}_j(u),
 \end{gather*}
 where $ \mathcal{D}_j(u)$ is obtained by differentiating a matrix $\mathcal{D}$ on $\mathbb{D}\big(u^0\big)$, whose diagonal entries only depend on~$u$:
 \begin{gather*}
 \mathcal{D}_j(u)=\frac{\partial \mathcal{D}(u)}{\partial u_j};
 \end{gather*}

 \item[$(B)$] $A(u)$ satisfies \eqref{23marzo2018-8}
 \begin{gather*}
 \frac{\partial A}{\partial u_j}=[\omega_j(0,u),A],\qquad j=1,\dots,n;
 \end{gather*}

 \item[$(C)$] the above non-linear system is Frobenius integrable;

 \item[$(D)$] the diagonal of $A$, namley the matrix $B=\operatorname{diag}(A_{11},\dots,A_{nn})$ in \eqref{11marzo2018-10-BIS}, is constant;

\item[$(E)$] given the relations \eqref{14marzo2018-4} for a fundamental matrix $Y(z,u)$ holomorphic on $\mathcal{R}\times \mathbb{D}\big(u^0\big)$, then each $H_r$ satisfies\footnote{We have
 $H_r(u)=\exp\{\operatorname{diag}({\mathcal{D}(u)})\} H_r^0$, with $H_r^0$ a constant invertible matrix.}
\begin{gather*}
\frac{\partial H_r}{\partial u_j}=\mathcal{D}_j(u)H_r\qquad \forall\, r\in\mathbb{Z};
\end{gather*}

\item[$(F)$] from $(E)$, it follows that $H_r(u)$ is holomorphic on $\mathbb{D}\big(u^0\big)$, and $Y_r(z,u)$ extends holomorphically on $\mathcal{R}\times \mathbb{D}\big(u^0\big)$.
\end{itemize}
\end{Proposition}

We prove the proposition in Appendix~\ref{appendixB}. We notice that in the above proposition we can choose an isomonodromic family $Y(z,u)=Y^{(0)}(z,u)$, given by a Levelt form~\eqref{11marzo2018-8} satisfying Proposition~\ref{30luglio 2018-3}. We also remark that by substitution of \eqref{11marzo2018-10-BIS} and \eqref{14marzo2018-1} into~\eqref{11marzo2018-1}, we find
\begin{gather}\label{18marzo2018-8}
 (F_1)_{ij}=\frac{A_{ij}}{u_j-u_i}, \qquad 1\leq i\neq j\leq n, \qquad (F_1)_{ii}=-\sum_{j\neq i}A_{ij}(F_1)_{ji}.
\end{gather}
Therefore
\begin{gather}\label{29marzo2018-2}
\omega_j(z,u)=zE_j+[F_1(u),E_j]+\mathcal{D}_j(u).
\end{gather}

\begin{Remark}In this paper, we have assumed that each $\omega_j(z,u)$ is holomorphic at $z=0$. More generally, one may study a Pfaffian system with Laurent expansions at $z=0$
 \begin{gather*}
 \omega_j(z,u)= \sum_{m=1}^{p_j}\frac{\omega_j^{(-m)}(u)}{z^m}+\sum_{m=0}^\infty \omega_j^{(m)}(u)z^m.
 \end{gather*}
 Here, we only remark that if $A(u)$ is non-resonant for each $u$ in the domain of interest (it means that the difference of two eigenvalues of~$A(u)$ cannot be a non-zero integer), then the Frobenius integrability conditions imply that $\omega_j(z,u)$ is holomorphic at $z=0$, namely
 \begin{gather*}
 \omega_j^{(-p_j)}(u)=\omega_j^{(-p_j+1)}(u)=\cdots =\omega_j^{(-1)}(u)=0.
 \end{gather*}
 To see this, substitute the Laurent expansion into~\eqref{24marzo2018-1}. We find the following recurrence relations:

From the negative powers of $z$:
\begin{gather}
 (A+p_j)\omega_j^{(-p_j)}-\omega_j^{(-p_j)}A=0; \label{eq(1)} \\
(A+m)\omega^{(-m)}-\omega^{(-m)}A=\big[\omega^{(-m-1)},\Lambda\big],\qquad \text{for $m= p_j-1,p_j-2,\dots,1$}; \label{eq(2)}\\
 \frac{\partial A}{\partial u_j}=\big[\omega^{(0)},A\big]+\big[\omega^{(-1)},\Lambda\big]. \label{eq(3)}
\end{gather}

 From the power $z^0$:
\begin{gather} A\omega_j^{(1)}-\omega_j^{(1)}A-\omega_j^{(1)}+E_j=\big[\omega_j^{(0)},\Lambda\big].\label{eq(4)}
\end{gather}

From the positive powers $z$:
\begin{gather}
(A-m-1)\omega_j^{(m+1)}-\omega_j^{(m+1)}A=\big[\omega_j^{(m)},\Lambda\big], \qquad m\geq 1 . \label{eq(5)}
 \end{gather}
 If $A$ is non-resonant, the Sylvester equation~\eqref{eq(1)} determines $\omega_j^{(-p_j)}=0$, thus the Sylvester equations~\eqref{eq(2)} yield $\omega_j^{(-m)}=0$, $ m= p_j-1,p_j-2,\dots,1$. Moreover, the equations~\eqref{eq(5)} determine~$\omega_j^{(m+1)}$ for $m\geq 1$ in terms of $\omega_j^{(1)}$. In particular, each $\omega_j^{(m+1)}=0 $ if $\omega_j^{(1)}$ is diagonal.

 Finally, we notice that, independently of the resonance properties of $A$, equation~\eqref{eq(4)} determines the off-diagonal entries of $\omega_j^{(0)}$ in terms of $\omega_j^{(1)}$, to be \begin{gather*}\big(\omega_j^{(0)}\big)_{ab}=(u_a-u_b)^{-1}\big(\big[\omega_j^{(1)},A\big]+\omega_j^{(1)}\big)_{ab}.\end{gather*} Moreover, the diagonal part of~\eqref{eq(4)} yields \begin{gather*}\big(\omega_j^{(1)}\big)_{aa}=\delta_{ja}-\sum_{b\neq a}\big(\big(\omega_j^{(1)}\big)_{ab}A_{ba}-A_{ab}\big(\omega_j^{(1)}\big)_{ba}\big).\end{gather*} Under the assumptions of Proposition~\ref{16marzo2018-1}, we find that $\omega_j^{(1)}=E_j$ and $\omega_j^{(m+1)}=0$ for $m\geq 1$.
 \end{Remark}

\section{Strong isomonodromic deformations}\label{19marzo2018-2}

In this section, we define strong isomonodromy deformations of the differential system~\eqref{11marzo2018-1}, which preserve a set of monodromy data, and we characterise the Pfaffian system responsible for them and its fundamental matrix solutions. The assumption here is that $\Lambda$ has distinct eigenvalues, namely we work on the domain $\mathbb{D}\big(u^0\big)$ previously introduced. In the next section, we will drop this assumption.

The notion of isomonodromic deformations given by Jimbo, Miwa and Ueno in~\cite{JMU} is stronger than the one defined in the previous section. It requires that a set of essential monodromy data, {\it not just} the monodromy matrices, are constant. In the standard theory of \cite{JMU}, the mat\-rix residue~$A(u)$ at a Fuchsian singularity must be non-resonant and reducible to a diagonal form with distinct eigenvalues, namely the Levelt form~\eqref{11marzo2018-8} is assumed to be $Y^{(0)}(z,u)=\widehat{Y}^{(0)}(z,u) z^{L_0(u)}$, where $L_0$ is a diagonal matrix with distinct eigenvalues (not differing by integers). Here, we do not assume this restriction, so introducing a first non-generic feature.

\begin{Definition}\label{18marzo2018-3}Let $Y(z,u)$ be a weak isomonodromic family of fundamental matrix solutions of the Pfaffian system \eqref{11marzo2018-2}, or equivalenlty \eqref{29marzo2018-1}, with coefficients \eqref{29marzo2018-2}. We call \eqref{11marzo2018-1} a~strong isomonodromic family of differential systems, and~$Y(z,u)$ a~strong isomonodromic family of fundamental solutions over $\mathbb{D}\big(u^0\big)$, if for all $r\in \mathbb{Z}$ the connection matrices~$H_r$ in~\eqref{14marzo2018-4} are independent of~$u$, namely
\begin{gather*} \dd H_r=0.
\end{gather*}
\end{Definition}

Recall that, by Proposition \ref{30luglio 2018-3}, a fundamental matrix $Y^{(0)}(z,u)$ in Levelt form \eqref{11marzo2018-8} satisfies the Pfaffian system. We have the following characterisation.
\begin{Proposition}\label{18marzo2018-4}
The deformation is strongly isomonodromic if and only if the fundamental matrices $Y_r(z,u)$ of the differential system \eqref{11marzo2018-1} also satisfy the Pfaffian system~\eqref{29marzo2018-1}
\begin{gather*}
\dd Y_r=\omega Y_r, \qquad \forall \, r\in\mathbb{Z},\end{gather*}
 being the coefficients
 \begin{gather}\label{18marzo2018-11}
 \omega_j(z,u)= zE_j+[F_1,E_j] \equiv zE_j+\left( \frac{A_{ab}(u)(\delta_{aj}-\delta_{bj})}{u_a-u_b} \right)_{a,b=1}^n.
\end{gather}
 \end{Proposition}

\begin{proof} Obvious. We have
\begin{gather*}
\omega =\dd Y Y^{-1}=\dd Y_r Y_r^{-1}+ Y_r \dd H_r H_r^{-1}Y_r^{-1}.
\end{gather*}
Therefore, $ \dd Y_r Y_r^{-1}=\omega$ if and only if $\dd H_r=0$. Since $\dd H_r=0$, then $\mathcal{D}_j=0$ in Proposition~\ref{16marzo2018-1}, so that $\omega_j$ is~\eqref{18marzo2018-11}.
\end{proof}

 It is a standard result that $\mathbb{S}_0$ and $\mathbb{S}_1$, together with $B$, suffice to generate $\mathbb{S}_r$ for any $r\in \mathbb{Z}$, through the well known formula \cite{BJL1}
 \begin{gather} \label{18marzo2018-6}
 \mathbb{S}_{2r}=e^{-2r\pi i B} \mathbb{S}_0 e^{2r\pi i B},\qquad \mathbb{S}_{2r+1}=e^{-2r\pi i B} \mathbb{S}_1 e^{2r\pi i B}.
 \end{gather}
We introduce connection matrices $C_r$ such that
\begin{gather*}
Y_r=Y^{(0)}C_r.
\end{gather*}
$C_0$ and the Stokes matrices suffice to generate all the matrices $C_r$. Indeed, $Y_r=Y_0\mathbb{S}_0\mathbb{S}_1\cdots \mathbb{S}_{r-1}$, so that we have
 \begin{gather} \label{18marzo2018-7}
C_r=C_0\mathbb{S}_0\mathbb{S}_1\cdots \mathbb{S}_{r-1}.
\end{gather}

\begin{Definition}\label{18marzo2018-10} The matrices $\mathbb{S}_r$, $B$, $D$, $L$, $C_r$ ($r\in \mathbb{Z}$) are called the {\rm essential monodromy data} associated with the fundamental matrix solutions $Y_r(z,u)$ and $Y^{(0)}(z,u)$ of the system \eqref{11marzo2018-1}. In view of \eqref{18marzo2018-6} and \eqref{18marzo2018-7}, it suffices to consider the data
 \begin{gather*}
 \mathbb{S}_0, \ \mathbb{S}_1, \ B, \ D, \ L, \ C_0.
 \end{gather*}
 or equivalently, for some fixed value of $r$,
 \begin{gather*}
 \mathbb{S}_r, \ \mathbb{S}_{r+1}, \ B, \ D, \ L, \ C_r.
 \end{gather*}
\end{Definition}
The above definition is similar to the definition of monodromy data given in \cite{JMU}, here including the case when $A$ may be resonant and/or non-diagonalizable.

For a Pfaffian system with $\omega_1, \dots, \omega_n$ holomorphic on $\mathbb{C}\times \mathbb{D}\big(u^0\big)$, by Propositions~\ref{30luglio 2018-3} and~\ref{16marzo2018-1}, we have that
\begin{gather*}
D,\ L, \ B \quad \text{are constant on $\mathbb{D}\big(u^0\big)$}.
\end{gather*}
The following characterisation of strong isomonodromic deformations in terms of essential monodromy data holds.

\begin{Proposition}\label{18marzo2018-9}A deformation is strongly isomonodromic as in Definition~{\rm \ref{18marzo2018-3}} if and only if
\begin{gather*}
\dd \mathbb{S}_r=\dd \mathbb{S}_{r+1}=0, \qquad \dd C_r=0,
\end{gather*}
for one value of $r\in\mathbb{Z}$.
\end{Proposition}

Before giving the proof, we make a few comments. First, we observe that if the Proposition holds for one value of $r$, for example for $\mathbb{S}_0$, $\mathbb{S}_1$ and $C_0$, then it holds for any $r$, and conversely, by formulae~\eqref{18marzo2018-6} and~\eqref{18marzo2018-7}.

Proposition \ref{18marzo2018-9} says that one can give a definition of strong isomonodromic deformations alternative to Definition~\ref{18marzo2018-3}, namely a deformation such that\footnote{We don't write $dD=0$, because $D$ has integer entries.}
\begin{gather*}
D \ \text{is constant}, \qquad \dd L=\dd C_0=\dd \mathbb{S}_0=\dd \mathbb{S}_1=\dd B=0
\end{gather*}
on $\mathbb{D}\big(u^0\big)$. This is the definition adopted in \cite{JMU}, in case $A(u)$ is diagonalizable with distinct eigenvalues and no resonances. Here we have extended it without assumptions on~$A(u)$. Notice also that one should, as in \cite{JMU}, say that $\dd\mathbb{S}_0=\dd\mathbb{S}_1=\dd C_0$ on the polydisc $\mathbb{D}^\prime$ where $Y_r(z,u)$ have the properties prescribed by Sibuya's results. Nevertheless, by Proposition~\ref{16marzo2018-1} point~(F), we can extend the properties to $\mathbb{D}\big(u^0\big)$.

In case $A(u)$ is skew-symmetric and diagonalizable, the above characterisation of isomonodromic deformations with form \eqref{18marzo2018-11} is well known in the theory of Frobenius manifolds, developed in \cite{Dub1,Dub2}, where $A(u)$ is named $V(u)$ and $\Lambda$ is called $U$.

\begin{proof}[Proof of Proposition \ref{18marzo2018-9}] For a weak isomonodromic deformation $D$, $L$ and $B$ are constant (Propositions~\ref{30luglio 2018-3} and~\ref{16marzo2018-1}), so we have nothing to prove about them.

$\bullet$ Suppose that the deformation is strong, so that $\dd Y_r=\omega Y_r$ by Proposition~\ref{18marzo2018-4}. We prove that $\dd C_r=0$ for any $r$. Indeed
\begin{align*}
\omega =\dd Y_rY_r^{-1}&=\dd Y^{(0)} \big(Y^{(0)}\big)^{-1}+Y^{(0)} \dd C_r C_r^{-1} \big(Y^{(0)}\big)^{-1}\\
& = \omega +Y^{(0)} \dd C_r C_r^{-1} \big(Y^{(0)}\big)^{-1} \quad \Longleftrightarrow \quad \dd C_r=0.
\end{align*}
We prove that $\dd \mathbb{S}_r=0$. Invoking again Proposition~\ref{18marzo2018-4}, we have
\begin{gather*}
\omega_j= \frac{\partial Y_{r+1}}{\partial u_j} Y_{r+1}^{-1}=\frac{\partial Y_{r}}{\partial u_j} Y_{r}^{-1} +
 Y_{r}\frac{\partial \mathbb{S}_r}{\partial u_j} \mathbb{S}_r^{-1}Y_{r}^{-1}
 =\omega_j+ Y_{r}\frac{\partial \mathbb{S}_r}{\partial u_j} \mathbb{S}_r^{-1}Y_{r}^{-1}
 \quad \Longleftrightarrow \quad
\frac{\partial \mathbb{S}_r}{\partial u_j}=0.
\end{gather*}

$\bullet$ Conversely, we assume that $\dd \mathbb{S}_0=\dd \mathbb{S}_1=\dd B=\dd C_0=\dd L=0$, and that $D$ is constant. First, we see that $\dd \mathbb{S}_r=0$ and $\dd C_r=0$ for any $r$, by formulae~\eqref{18marzo2018-6} and~\eqref{18marzo2018-7}. By virtue of Proposition \ref{18marzo2018-4}, it suffices to show that $\dd Y_r=\omega Y_r$ for any~$r$. By construction $\partial_zY_r=\omega_0(z,u)Y_r$, which is the differential system~\eqref{11marzo2018-1}, so it suffices to show that $\partial_j Y_r=\omega_j Y_r$, $j=1,\dots,n$. Since $\dd C_0=0$ and each $\dd \mathbb{S}_r=0$, we have
\begin{gather}\label{2agosto2018-1}
\dd Y^{(0)}\big(Y^{(0)}\big)^{-1} = \dd Y_0 Y_0^{-1}=\cdots= \dd Y_r Y_r^{-1}=\dd Y_{r+1}Y_{r+1}^{-1}=\cdots \qquad \forall\, r\in \mathbb{Z}.
\end{gather}
Moreover, $Y^{(0)}$ is a weak isomonodromic family, so that $\dd Y^{(0)}=\omega Y^{(0)}$, where $\omega$ is \eqref{29marzo2018-1} with components \eqref{29marzo2018-2}, namely
\begin{gather*}
\omega_j(z,u)=zE_j+[F_1(u),E_j]+\mathcal{D}_j(u).
\end{gather*}
Thus, by \eqref{2agosto2018-1}, each $Y_r$ also satisfies
\begin{gather*}
\dd Y_r= \omega Y_r,\qquad\forall\, r.
\end{gather*}
We need to prove that $\mathcal{D}_j(u)=0$. At $z=\infty$, using the fact that $dB=0$ and $\partial_j\Lambda=E_j$, we have for any $r\in \mathbb{Z}$
 \begin{gather*}
\omega_j=\frac{\partial Y_r}{\partial u_j}Y_r^{-1}=\frac{\partial \widehat{Y}_r}{\partial u_j}\widehat{Y}_r^{-1} +z\widehat{Y}_r E_j \widehat{Y}_r ^{-1}= zE_j +[F_1,E_j]+\mathcal{O}\big(z^{-1}\big).
\end{gather*}
The above $\mathcal{O}(z^{-1})$ stands for an asymptotic expansion in $\mathcal{S}_r(\overline{\mathbb{D}^\prime})$, given for any $r$ by the same series in $z^{-1}$. Since $ \omega_j$ is single valued in $z$, then $\mathcal{O}(z^{-1})$ is a convergent Taylor series at $z=\infty$. Thus, $\omega_j-zE_j-[F_1,E_j]$ is holomorphic both at $z=0$ and $z=\infty$, and it vanishes as $z\to \infty$. Hence, by Liouville theorem,
 \begin{gather*}
\omega_j =zE_j+[F_1,E_j],
 \end{gather*}
with $\mathcal{D}_j=0$. The equality extends from $\mathbb{D}^\prime$ to $\mathbb{D}\big(u^0\big)$ by analiticity.
\end{proof}

\section[The non-generic case of coalescing eigenvalues of $\Lambda(u)$]{The non-generic case of coalescing eigenvalues of $\boldsymbol{\Lambda(u)}$}\label{25marzo2018-3}
Having defined weak and strong isomonodromic deformations when the eigenvalues of $\Lambda$ are distinct and $A(u)$ is any, the next step towards non-generic isomonodromic deformations is to extended weak and strong deformations to the case when some eigenvalues of $\Lambda$ coalesce, namely when $u_i-u_j\to 0$ for some $i\neq j$. In this section, we give a ``holomorphic'' extension, summarised in Theorem~\ref{25marzo2018-1} below, which constitutes one of the main results of~\cite{CDG}.

 Let $u^C=\big(u_1^C,\dots,u_n^C\big)\in\mathbb{C}^n$ be a {\it coalescence point} (here, ``$C$'' stands for ``coalescence''), namely \begin{gather*} u_i^C=u_j^C \quad \text{for some indexes } i \neq j\in\{1,2,\dots,n\}.
 \end{gather*}
 We consider a polydisc of radius $\epsilon>0$ centered at $u^C$ and denote it by $\mathcal{U}_\epsilon\big(u^C\big)$. It is to be noticed that there exists a {\it coalescence locus} in $\mathcal{U}_\epsilon\big(u^C\big)$, let it be denoted by $\Delta$, containing $u^C$ and defined by
 \begin{gather*}
 \Delta:=\mathcal{U}_\epsilon\big(u^C\big)\cap \left(\bigcup_{i\neq j}\{u_i-u_j=0\}\right).
\end{gather*}
In order to study the local theory at $u^C$, we assume that $\epsilon$ is small, so that $u^C$ is ``the most'' coalescent point. This means that if $k\neq l$ are indexes such that $u_k^C-u_l^C\neq 0$, then $\epsilon$ is sufficiently small to guarantee that $u_k-u_l\neq 0$ for every point of~$\mathcal{U}_\epsilon\big(u^C\big)$.

We fix a half-line of direction $\arg z=\widetilde{\tau}$ in $\mathcal{R}$ which, now, does not coincide with any of the Stokes rays associated with $\Lambda\big(u^C\big)$, namely with the half lines in $\mathcal{R}$ specified by
\begin{gather*}
\operatorname{Re}\big(z\big(u_k^C-u_l^C\big)\big)=0, \qquad \operatorname{Im}\big(z\big(u_k^C-u_l^C\big)\big)<0,
\end{gather*}
for those $k,l$ such that $u_k^C-u_l^C\neq 0$. We call $\widetilde{\tau}$ an {\it admissible direction at $u^C$}. The choice of $\widetilde{\tau}$ determines a {\it cell decomposition} of $\mathcal{U}_\epsilon\big(u^C\big)$, which is based on two ingredients. One is $\Delta$ above. The other one is the so called ``crossing locus''. In order to describe it, observe that if $u\in \mathcal{U}_\epsilon\big(u^C\big)$ (points of $\Delta$ are not excluded) is such that $u_i\neq u_j$ for some $i\neq j$, then the (infinitely many in~$\mathcal{R}$) Stokes rays
\begin{gather*}
\operatorname{Re}(z(u_i-u_j))=0,\qquad \operatorname{Im}(z(u_i-u_j))<0,
\end{gather*}
 corresponding to $(u_i,u_j)$, are well defined. The {\it crossing locus} $X(\widetilde{\tau})$ is made of those points such that some Stokes rays ``cross'' the admissible rays $\{z\in\mathcal{R}\,|\,\arg z=\widetilde{\tau}+h\pi\}$, $ h\in\mathbb{Z}$. Namely, $X(\widetilde{\tau})$ is made of points $u$ such that $\operatorname{Re}(e^{i\widetilde{\tau}}(u_i-u_j))=0$ for some $u_i\neq u_j$. Precisely,
\begin{gather*}
X(\widetilde{\tau}):=\bigcup_{1\leq i < j\leq n} \left\{ u\in \mathcal{U}_\epsilon\big(u^C\big) \text{ such that } u_i\neq u_j \text{ and } \arg(u_i-u_j)=\frac{3\pi}{2}-\widetilde{\tau} \text{ mod }\pi \right\}.
\end{gather*}
Let the ``walls'' be defined as
\begin{gather*}W(\widetilde{\tau}):=\Delta\cup X(\widetilde{\tau}).
\end{gather*}
Following \cite{CDG}, every connected component of $ \mathcal{U}_\epsilon\big(u^C\big)\backslash W(\widetilde{\tau})$ is called a {\it $\widetilde{\tau}$-cell}. We have proved in \cite{CDG} that every $\widetilde{\tau}$-cell is a topological cell, so in particular it is simply connected (simple connectedness is important for the proof, given in~\cite{CDG}, of Proposition~\ref{22marzo2018-1-BIS} below).

The isomonodromy deformation theory can be extended, in a holomorphic way, to the case of coalescing eigenvalues when a certain vanishing condition holds for the entries of~$A(u)$.

\begin{Lemma} \label{19marzo2018-5} Consider the Pfaffian system \eqref{11marzo2018-2}--\eqref{11marzo2018-6}, with $A(u)$ holomorphic on $\mathcal{U}_\epsilon\big(u^C\big)$. Let~$\mathbb{D}\big(u^0\big)$ be compactly contained in a $\widetilde{\tau}$-cell. Assume that the Frobenius integrability holds on~$\mathbb{D}\big(u^0\big)$ and that the $\omega_j(z,u)$ are holomorphic on $\mathbb{C}\times \mathbb{D}\big(u^0\big)$, with $z=\infty$ being at most a~pole, so that Proposition~{\rm \ref{16marzo2018-1}} holds on $\mathbb{D}\big(u^0\big)$ and \eqref{11marzo2018-2}--\eqref{11marzo2018-6} becomes system~\eqref{29marzo2018-1} with coefficients~\eqref{29marzo2018-2} on~$\mathbb{D}\big(u^0\big)$. Then, the following statements are true:
\begin{itemize}\itemsep=0pt
\item[$i)$] The coefficients $\omega_j(z,u)$ in~\eqref{29marzo2018-2} extend holomorphically on $\mathcal{U}_\epsilon\big(u^C\big)$ if and only if $\mathcal{D}_j(u)$ is holomorphic on $\mathcal{U}_\epsilon\big(u^C\big)$ and the following vanishing conditions hold in $\mathcal{U}_\epsilon\big(u^C\big)$:
\begin{gather}\label{22marzo2018-1}
A_{ij}(u)=\mathcal{O}(u_i-u_j)\qquad \text{whenever $u_i-u_j\to 0$ at a point of $\Delta$}.
\end{gather}
 In this case, the system \eqref{29marzo2018-1}, \eqref{29marzo2018-2} is Frobenius integrable on the whole $\mathcal{U}_\epsilon\big(u^C\big)$.

\item[$ii)$] A fundamental matrix solution $Y(z,u)$ exists holomorphic on $\mathcal{R}\times \mathcal{U}_\epsilon\big(u^C\big)$ if and only if~\eqref{22marzo2018-1} holds.
\end{itemize}
 \end{Lemma}

\begin{proof} The statement follows simply by recalling that, by Proposition \ref{16marzo2018-1}, for $u\in\mathbb{D}\big(u^0\big)$ we have \eqref{29marzo2018-2}, namely
\begin{gather*}
 \omega_j(z,u)= zE_j+\left( \frac{A_{ab}(\delta_{aj}-\delta_{bj})}{u_a-u_b} \right)_{a,b=1}^n+\mathcal{D}_j(u),
\end{gather*}
so that if $u_i-u_j\to 0$, the condition $A_{ij}\to 0$ guarantees analyticity on $\mathcal{U}_\epsilon\big(u^C\big)$. The integrability condition \eqref{11marzo2018-3} then holds by analytic continuation from $\mathbb{D}\big(u^0\big)$ to $\mathcal{U}_\epsilon\big(u^C\big)$. The last statement concerning $Y(z,u)$ follows from the fact that if~\eqref{22marzo2018-1} holds, then the Pfaffian system is integrable and linear on $\mathcal{U}_\epsilon\big(u^C\big)$ with holomorphic (in $u$) coefficients. Conversely, if a fundamental matrix $Y(z,u)$ exists holomorphic on $\mathcal{R}\times \mathcal{U}_\epsilon\big(u^C\big)$, then $\omega=\dd YY^{-1}$ has coefficients holomorphic in $u$ on $\mathcal{U}_\epsilon\big(u^C\big)$, which implies~\eqref{22marzo2018-1}.
\end{proof}

Under the assumptions of Lemma \ref{19marzo2018-5}, then Propositions \ref{11marzo2018-7} and~\ref{30luglio 2018-3} hold with the replacements $u^0\mapsto u^C$ and $\mathbb{D}\big(u^0\big)
\mapsto \mathcal{U}_\epsilon\big(u^C\big)$, if and only if $A_{ij}(u)=\mathcal{O}(u_i-u_j)$ for $u_i-u_j\to 0$ at $\Delta$. Namely

\begin{Corollary}\label{31luglio2018-1} Under the assumptions of Lemma~{\rm \ref{19marzo2018-5}}, $A(u)$ is holomorphically similar to a~constant Jordan form~$J$ on $\mathcal{U}_\epsilon\big(u^C\big)$, and there exists an isomonodromic family of fundamental solutions $Y^{(0)}(z,u)$ in Levelt form \eqref{11marzo2018-8}--\eqref{3aprile2018-1}, holomorphic on $\mathcal{R}\times \mathcal{U}_\epsilon\big(u^C\big)$, with constant~$L$ and~$D$, if and only if $A_{ij}(u)=\mathcal{O}(u_i-u_j)$ for $u_i-u_j\to 0$ at $\Delta$.
\end{Corollary}

We turn to the fundamental matrices $Y_r(z,u)$ in \eqref{11marzo2018-10-BIS}--\eqref{11marzo2018-10}. Notice that $\widetilde{\tau}$ is admissible at any $u\in\mathbb{D}\big(u^0\big)$, because $\mathbb{D}\big(u^0\big)$ is assumed to be compactly contained in a $\widetilde{\tau}$-cell. Then, the matrices $Y_r(z,u)$ are well defined and holomorphic on a small $\mathbb{D}^\prime\subset \mathbb{D}\big(u^0\big)$, as proved by Sibuya, and all the results described in Section~\ref{19marzo2018-1} apply on $\mathbb{D}^\prime$. Moreover, by item (F) of Proposition~\ref{16marzo2018-1}, $\mathbb{D}^\prime$ is extended to $\mathbb{D}\big(u^0\big)$, so that the fundamental matrices~$Y_r(z,u)$ are holomprphic on $\mathcal{R}\times \mathbb{D}\big(u^0\big)$. Complications arise if we want to study the $u-$analytic continuation of the matrices $Y_r(z,u)$ on the whole $\mathcal{U}_\epsilon\big(u^C\big)$, and their asymptotic behaviour.

The problem can indeed be studied also in the non-isomonodromic case, when the $Y_r(z,u)$'s are defined only for $u\in \mathbb{D}^\prime$. When $u$ moves outside $\mathbb{D}^\prime$, Sibuya's local result concerning
 analyticity in $u$ does no longer apply. This issue was analyszed in \cite{CDG}, for a system~\eqref{11marzo2018-1} without any assumption that it be isomonodromic, namely without assuming that there is a Pfaffian system behind \eqref{11marzo2018-1}. The following holds (indeed, for the more general system \eqref{22marzo2018-6} below).

 \begin{Proposition}[\cite{CDG}] \label{22marzo2018-1-BIS} Let $A(u)$ be holomorphic on $\mathcal{U}_\epsilon\big(u^C\big)$. For any $z\in\mathcal{R}$, the fundamental matrices $Y_r(z,u)$, $r\in\mathbb{Z}$, of a differential system~\eqref{11marzo2018-1} or~\eqref{22marzo2018-6}, not necessarily isomonodromic, defined on $\mathbb{D}^\prime\subset \mathbb{D}\big(u^0\big)$ $($on $\mathbb{D}\big(u^0\big)$, in the isomonodromic case$)$ can be analytically continued w.r.t.~$u$ on the whole $\widetilde{\tau}$-cell containing $u^0$, maintaining the asymptotics \eqref{11marzo2018-10}. The asymptotics is uniform in any compact subset $K$ of the cell, for $z\to \infty$ in the following sector \begin{gather*}\mathcal{S}_r(K):=\bigcap_{u\in K}\mathcal{S}_r(u).\end{gather*}
 Moreover, the analytic continuation -- maintaining for each $u$ the asymptotics in $\mathcal{S}_r(u)$ -- exists along any curve slightly beyond the boundary of the cell, if the curve crosses the boundary at a~point corresponding to just one Stokes ray crossing $\arg z=\widetilde{\tau}$ $($simple crossing$)$.
\end{Proposition}

 We refer to \cite{CDG} for the proof. Here we notice that in the strong isomonodromic case of system~\eqref{11marzo2018-1}, it is easy to show the existence of the analytic continuation claimed in Proposition~\ref{22marzo2018-1-BIS}. Indeed, if $A(u)$ is holomorphic on $\mathcal{U}_\epsilon\big(u^C\big)$, the analytic continuation exists beyond the $\widetilde{\tau}$-cell containing $u^0$, over any simply connected subset of $\mathcal{U}_\epsilon\big(u^C\big)\backslash \Delta$ containing~$u^0$, because
 \begin{gather*}
 \omega_j(z,u)= zE_j+\left( \frac{A_{ab}(\delta_{aj}-\delta_{bj})}{u_a-u_b} \right)_{a,b=1}^n
\end{gather*}
 is holomorphic on $\mathcal{U}_\epsilon\big(u^C\big)\backslash \Delta$. Therefore, since $\dd Y_r=\omega Y_r$, then $Y_r$ enjoys the properties of the solutions of a~linear Pfaffian system with holomorphic coefficients away form $\Delta$. However, the non-isomonodromic case (system~\eqref{11marzo2018-1} or~\eqref{22marzo2018-6}) requires a more sophisticated extension of Sibuya's results. Moreover, the proof that the analytic continuation of $Y_r(z,u)$ maintains the asymptotic behaviour $Y_F(z,u)$ on $\mathcal{S}_r(K)$ is non trivial, both in the non-isomonodromic and isomonodromic cases. See \cite[Chapters~12 and~13]{CDG} for details.

Having established in Proposition \ref{22marzo2018-1-BIS} the analytic continuation and the asymptotics on a~$\widetilde{\tau}$-cell, we have to study what happens if $u$ moves outside the cell along a curve crossing~$W(\widetilde{\tau})$. We face further problems:
 \begin{itemize}\itemsep=0pt
 \item[i)] When the crossing occurs at a point of $\Delta$, the coefficients $F_k(u)$ in $Y_F(z,u)$ have poles for $u_i-u_j=0$, because (see~\cite{CDG})
 \begin{gather} \label{5aprile2018-1}
(F_k)_{i j}=\frac{1}{u_j-u_i}\left\{
(A_{i i}-A_{jj}+k-1
)(F_{k-1})_{i j}
+\sum_{p\neq i}A_{i p}(F_{k-1})_{p j}\right\}, \qquad i\neq j;\\
\label{5aprile2018-2}
 k(F_{k})_{i i}=-\sum_{j\neq i }A_{i j}(F_{k})_{ji}.
\end{gather}

\item[ii)] An actual solution $Y_r(z,u)$ is in general multivalued for loops around $u_i-u_j=0$ and diverges as $u_i-u_j\to 0$ along any direction (provided we are in a case when $Y_r(z,u)$ can be extended analytically along any curve outside the cell, as it certainly happens for strong isomonodromic deformations, as discussed in the comments after Proposition~\ref{22marzo2018-1-BIS}). The reader can find some explicit examples in the papers~\cite{Eretico1} and~\cite{Alcala}, respectively based on talks at workshops in CRM, Pisa (February 13--17, 2017) and at the University of Alcal\'a (September 4--8, 2017).

\item[iii)] The asymptotic behaviour of $Y_r(z,u)$ in sectors $\mathcal{S}_r(K)$ does no longer hold when $u$ is outside the cell containing $u^0$, because some Stokes rays cross the admissible direction~$\widetilde{\tau}$ (again, provided that $Y_r(z,u)$ can be extended analytically to some simply connected subset of~$\mathcal{U}_\epsilon\big(u^C\big)\backslash \Delta$ containing the cell, as it happens in the case of strong isomonodromic deformations).
\end{itemize}

The issue above can be solved in the {\it strong} isomonodromic case when the vanishing conditions~\eqref{22marzo2018-1} hold. Let the deformation be strongly isomonodromic in $\mathbb{D}\big(u^0\big)$ compactly contained in a~$\widetilde{\tau}$-cell, so that
\begin{gather}\label{22marzo2018-3}
\dd Y_r=\omega Y_r
\end{gather}
on $\mathbb{D}\big(u^0\big)$, with $\omega$ as in \eqref{18marzo2018-11}. Since \eqref{22marzo2018-3} is linear, and its coefficients are holomorphic on the whole $\mathcal{U}_\epsilon\big(u^C\big)$ if and only if conditions~\eqref{22marzo2018-1} hold, then the following proposition holds.

\begin{Proposition} \label{23marzo2018-1} In case of strong isomonodromy deformations, the fundamental mat\-ri\-ces $Y_r(z,u)$ extend to holomorphic functions on $\mathcal{R}\times\mathcal{U}_\epsilon\big(u^C\big)$ if and only if the vanishing conditions~\eqref{22marzo2018-1} hold.
 \end{Proposition}

As far as the formal fundamental matrix $Y_F(z,u)$ is concerned, we have the following

\begin{Proposition} \label{23marzo2018-2} Let $A(u)$ be holomorphic on $\mathcal{U}_\epsilon\big(u^C\big)$ and assume that the vanishing conditions~\eqref{22marzo2018-1} hold in $\mathcal{U}_\epsilon\big(u^C\big)$. Then, in case of strong isomonodromic deformations, the coefficients~$F_k(u)$, $k\geq 1$, of the formal solution $Y_F(z,u)$ are holomorphic on $\mathcal{U}_\epsilon\big(u^C\big)$.
\end{Proposition}

\begin{proof} The coefficients $F_k(u)$ are computed recursively from \eqref{11marzo2018-1}. This standard computation yields coefficients depending rationally on $A(u)$ and on the differences $(u_i-u_j)$, which appear in the denominators, as in \eqref{5aprile2018-1}. In particular \eqref{18marzo2018-8} holds:
\begin{gather*}(F_1)_{ij}(u)=\frac{A_{ij}(u)}{u_j-u_i}, \qquad i\neq j, \qquad (F_1)_{ii}(u)=-\sum_{j\neq i}A_{ij}(u)(F_1)_{ji}(u).
\end{gather*}
Thus, if \eqref{22marzo2018-1} holds, $F_1(u)$ is holomorphic in $\mathcal{U}_{\epsilon}\big(u^C\big)$. By Proposition \ref{22marzo2018-1-BIS}, the asymptotic expansion is uniform in compact subsets of a cell, so we can substitute it into $\partial_iY=\omega_i Y$ (with~$\omega_i$ as in~\eqref{18marzo2018-11}), compare coefficients of $z^{-l}$ and find
\begin{gather}\label{21luglio-3}
[F_{l+1}(u),E_i]=[F_1(u),E_i]F_l(u)-\partial_i F_l(u),\qquad l\geq 1,
\end{gather}
with
\begin{gather*}
 [F_{l+1}(u),E_i ]=\left(
\begin{matrix}
0 & & (F_{l+1})_{1i} & & 0
\\
 & & \vdots & &
\\
-(F_{l+1})_{i1} & \cdots & 0 & \cdots & -(F_{l+1})_{in}
\\
 & & \vdots &
 \\
 0 & & (F_{l+1})_{ni} & & 0
\end{matrix}
\right).
\end{gather*}
Moreover, $\operatorname{diag}(F_{l+1})$ is determined by
\begin{gather}\label{3nov2015-6}
l (F_{l+1})_{ii}(u)=-\sum_{j\neq i}A_{ij}(u)(F_l)_{ji}(u).
\end{gather}
Therefore, \eqref{21luglio-3}--\eqref{3nov2015-6} recursively determines $F_{l+1}$ as a function of $F_l, F_{l-1}, \dots, F_1$. Since $F_1$ is holomorphic when conditions~\eqref{22marzo2018-1} hold, by induction all the $F_{l+1}(u)$ are holomorphic.
 \end{proof}

\begin{Remark}Proposition \ref{23marzo2018-2} holds only in the strong isomonodromic case. In the non-isomonodromic case the vanishing conditions~\eqref{22marzo2018-1} only guarantee that~$F_1(u)$ is holomorphic. In order for $F_1, F_2, \dots, F_l$ to be holomorphic up to a certain $l$, also the following quantities
\begin{gather*}\bigl( A_{ii}-A_{jj}+k-1 \bigr)(F_{k-1})_{ij}+\sum_{p\neq i}A_{i p}(F_{k-1})_{p j},
\qquad 2\leq k \leq l,\end{gather*}
must vanish when $u_i-u_j\to 0$. The above, which follow from \eqref{5aprile2018-1}--\eqref{5aprile2018-2}, are conditions on the entries of~$A$, since the $F_k$ are determined by~$A$. \end{Remark}

It remains to check what happens to the asymptotic behaviour of the matrices $Y_r(z,u)$ outside the cell. This requires a certain amount of non-trivial work, which we have done in~\cite{CDG} for a~system of the form
 \begin{gather} \label{22marzo2018-6}
 \frac{\dd Y}{\dd z}=\left(\Lambda +\frac{A(u)}{z}+\sum_{j=2}^\infty \frac{A_j(u)}{z^j}\right) Y, \qquad u\in \mathcal{U}_\epsilon\big(u^C\big),
 \end{gather}
 without assuming that the system is isomonodromic. The series above is assumed to converge at $z=\infty$ with holomorphic matrix coefficients $A_j(u)$ on $\mathcal{U}_\epsilon\big(u^C\big)$. The asymptotic theory at $z=\infty$ for \eqref{22marzo2018-6}, with $u\in\mathbb{D}^\prime$ sufficiently small contained in a $\widetilde{\tau}$-cell, is the same as for~\eqref{11marzo2018-1}. Namely, there is a unique formal solution \begin{gather*}
Y_F(z,u)=F(z,u) z^{B(u)}e^{z\Lambda} ,\qquad B(u)=\operatorname{diag}(A(u)), \qquad F(z,u)=I+\sum_{k=1}^\infty F_k(u)z^{-k},
\end{gather*}
and unique actual solutions $Y_r(z,u)$.

 In order to proceed, we need to take $\epsilon$ sufficiently small, as follows. Consider the sub-class of Stokes rays associated with the pairs of eigenvalues $u_i$ and $u_j$, with label $i$, $j$ corresponding to components of $u^C$ satisfying $u_i^C\neq u_j^C$. If $\epsilon$ is small enough\footnote{It suffices to take $\epsilon$ less than the minimum over $i$ and $j$, such that $u_i^C\neq u_j^C$, of the distances between the two parallel lines in the complex plane, one passing through~$u_i^C$ and one through~$u_j^C$, with direction $3\pi/2-\widetilde{\tau}$ (mod~$\pi$, or mor~$2\pi$, which is the same).}, these rays do not cross any admissible direction $\widetilde{\tau} +h\pi$ when $u$ varies in $\mathcal{U}_\epsilon\big(u^C\big)$. We define a sector $\widehat{\mathcal{S}}_r(u)\subset \mathcal{R}$, which contains the ``half-plane'' $\widetilde{\tau}+(r-2)\pi<\arg z <\widetilde{\tau}+(r-1)\pi$ and extends up to the nearest Stokes rays lying outside the ``half-plane'' in the sub-class above (notice that $\mathcal{S}_r(u)\subset \widehat{\mathcal{S}}_r(u)$). Then, we define
 \begin{gather*}
 \widehat{\mathcal{S}}_r:=\bigcap_{u\in \mathcal{U}_\epsilon(u^C)} \widehat{\mathcal{S}}_r(u),
 \end{gather*}
 The sectors $\widehat{\mathcal{S}}_r$ have angular opening greater than $\pi$. The following theorem has been proved in~\cite{CDG}. It requires a non-trivial amount of work, which we necessarily skip here.

 \begin{Theorem}[\cite{CDG}] \label{23marzo2018-3} Consider the differential system \eqref{22marzo2018-6}, with coefficients $A(u)$, $A_j(u)$ holomorphic on $\mathcal{U}_\epsilon\big(u^C\big)$, where $\epsilon$ is specified as above. Assume that all the $F_k(u)$ are holomorphic on $\mathcal{U}_\epsilon\big(u^C\big)$. Moreover, assume that the fundamental matrices
 \begin{gather*}Y_r(z,u) =\widehat{Y}_r(z,u) z^{B(u)}e^{z\Lambda},
 \end{gather*}
 which are holomorphic on a $\widetilde{\tau}$-cell by Proposition~{\rm \ref{22marzo2018-1-BIS}}, admit analytic continuation on the who\-le~$\mathcal{U}_\epsilon\big(u^C\big)$ as single valued holomorphic functions of~$u$, for any $r\in \mathbb{Z}$ and $z$ fixed. Then the following results hold.

 \begin{itemize}\itemsep=0pt
 \item The asymptotic representation of $Y_r(z,u)$ in terms of the formal matrix
 \begin{gather*} Y_F(z,u)=F(z,u)z^{B(u)}e^{z\Lambda},\end{gather*} extends beyond the $\widetilde{\tau}$-cell, namely
\begin{gather*}
 \widehat{Y}_r(z,u)\sim F(z,u)= I+\sum_{k=1}^\infty F_k(u) z^{-k}, \qquad z\to \infty \text{ in } \widehat{\mathcal{S}}_r,
 \end{gather*}
 uniformly in every compact subset of $\mathcal{U}_{\epsilon_1}\big(u^C\big)$ for every $\epsilon_1<\epsilon$. Moreover, for any $r\in\mathbb{Z}$ the Stokes matrix $\mathbb{S}_r(u)$ relating $Y_r(z,u)$ and $Y_{r+1}(z,u)$ satisfies
 \begin{gather*}
 (\mathbb{S}_r)_{ij}=(\mathbb{S}_r)_{ji}=0 \qquad \text{for $i$, $j$ such that } u_i^C=u_j^C.
 \end{gather*}
 \item The system \eqref{22marzo2018-6} at {\rm fixed} $u=u^C$, namely
 \begin{gather} \label{29marzo2018-3}
 \frac{\dd Y}{\dd z}=\left(\Lambda\big(u^C\big) +\frac{A\big(u^C\big)}{z}+\sum_{j=2}^\infty \frac{A_j\big(u^C\big)}{z^j}\right) Y,
 \end{gather}
 has formal solutions with the following structure \footnote{Without the assumptions of the theorem, the formal solutions have a more complicated structure. See \cite{BJL2,BJL3} for a general theory, and \cite{CDG} for the specific case here studied.}
 \begin{gather*}
 \mathring{Y}_F(z)=\mathring{F}(z) z^{B(u^C)}e^{z\Lambda(u^C)},
 \end{gather*}
 where $\mathring{F}(z)=I+ \sum\limits_{f=1}^\infty \mathring{F}_Kz^{-k} $ is a formal series and $ B\big(u^C\big)=\operatorname{diag}\big(A_{11}\big(u^C\big),\dots,A_{nn}\big(u^C\big)\big)$.

 \item The solution $\mathring{Y}_F(z)$ is unique if and only if the diagonal entries of $A\big(u^C\big)$ do not differ by non-zero integers; otherwise, there is a family of formal solutions $ \mathring{Y}_F(z)$.

 \item For any fixed formal solution $\mathring{Y}_F(z)$ of \eqref{29marzo2018-3}, there are unique actual solutions
 \begin{gather*}
 \mathring{Y}_r(z)=\mathring{\widehat{Y}}_r(z) z^{B(u^C)}e^{z\Lambda(u^C)},\qquad r\in \mathbb{Z},
 \end{gather*} such that
 \begin{gather*}
\mathring{\widehat{Y}}_r(z)\sim \mathring{F}(z) \qquad\text{for } z\to \infty \text{ in } \widehat{\mathcal{S}}_r\big(u^C\big).
 \end{gather*}
\item In particular, there exists a formal solution $ \mathring{Y}_F(z)$ of \eqref{29marzo2018-3} satisfying
 \begin{gather*}
 \mathring{Y}_F(z)=Y_F\big(z,u^C\big).
 \end{gather*}
 The corresponding unique actual solutions $ \mathring{Y}_r(z)$ satisfy
 \begin{gather*}
 \mathring{Y}_r(z)= \lim_{u\to u_c}Y_r(z,u)\equiv Y_r\big(z,u^C\big).
 \end{gather*}
Let $\mathring{\mathbb{S}}_r$ be the Stokes matrices of the above solutions $\mathring{Y}_r(z)$. Then,
 \begin{gather*}
 \lim_{u\to u_c} \mathbb{S}_r(u)=\mathring{\mathbb{S}}_r.
 \end{gather*}
\end{itemize}
\end{Theorem}

\begin{Remark}We use the notation $\mathring{Y}$, $\mathring{F}$, $\mathring{\mathbb{S}}$, etc., for objects $Y$, $F$, $\mathbb{S}$, etc, relative to the system at $u=u^C$, such as~\eqref{29marzo2018-3} above and~\eqref{23marzo2018-6} below.
\end{Remark}

Theorem \ref{23marzo2018-3} is at the core of the validity of one of the main results of \cite{CDG}, namely Theo\-rem~\ref{25marzo2018-1} below. Indeed, by Propositions \ref{23marzo2018-1} and \ref{23marzo2018-2}, the assumptions of Theorem \ref{23marzo2018-3} hold in the strong isomonodromic case if the vanishing conditions \eqref{22marzo2018-1} hold. This yields the following

\begin{Theorem}[\cite{CDG}]\label{25marzo2018-1}
Let $A(u)$ be holomorphic on $\mathcal{U}_\epsilon\big(u^C\big)$ and $\epsilon$ small as specified above. If the system
\begin{gather}
\label{11marzo2018-TRIS}
\partial_z Y=\left(\Lambda+\frac{A(u)}{z}\right)Y.
\end{gather}
is strongly isomonodromic on a polydisk interior to a $\widetilde{\tau}$-cell and the vanishing conditions \eqref{22marzo2018-1} hold, then:
\begin{enumerate}\itemsep=0pt
\item[$a)$] the coefficients of the unique formal solution $Y_F(z,u)=F(z,u)z^Be^{z\Lambda}$ are holomorphic on $\mathcal{U}_\epsilon\big(u^C\big)$, and the corresponding actual solutions $Y_r(z,u)$ extend holomorphically on \smash{$\mathcal{R}\times \mathcal{U}_{\epsilon}\big(u^C\big)$}, maintaining the asymptotics
\begin{gather*}
\widehat{Y}_r(z,u)\sim F(z,u) \qquad \text{for $z\to\infty$ in $\widehat{\mathcal{S}}_r$},
\end{gather*}
 uniformly in every compact subset of $\mathcal{U}_{\epsilon_1}\big(u^C\big)$ $(\forall\, \epsilon_1<\epsilon)$.

\item[$b)$] For any $r$, the essential monodromy data $\mathbb{S}_r$, $\mathbb{S}_{r+1}$, $B$, $L$, $D$ and $C_r$ are constant on the whole $\mathcal{U}_\epsilon\big(u^C\big)$ and satisfy
\begin{gather*}
(\mathbb{S}_r)_{ij}=(\mathbb{S}_{r+1})_{ij}=(\mathbb{S}_r)_{ji}=(\mathbb{S}_{r+1})_{ji}=0 \qquad \text{for $i$, $j$ such that } u_i^C=u_j^C.
\end{gather*}

\item[$c)$] They coincide with a set of essential monodromy data of
\begin{gather}\label{23marzo2018-6}
\partial_z Y=\left(\Lambda\big(u^C\big)+\frac{A\big(u^C\big)}{z}\right)Y.
\end{gather}
 in the following way. Take the formal solution $\mathring{Y}_F(z)$ of \eqref{23marzo2018-6} satisfying
 \begin{gather*}
 \mathring{Y}_F(z)=Y_F\big(z,u^C\big)
 \end{gather*} and the associated actual solutions $\mathring{Y}_r(z)$, with Stokes matrices $\mathring{\mathbb{S}}_r$. Then, choose a fundamental matrix solution of \eqref{23marzo2018-6} in Levelt form
 \begin{gather}
 \label{31luglio2018-7}
 \mathring{Y}^{(0)}(z)= \mathring{\widehat{Y}}(z)z^D z^{\mathring{L}},
\end{gather}
 and take the corresponding connection matrix $\mathring{C}_r$ such that
 \begin{gather*}\mathring{Y}_r(z)=\mathring{Y}^{(0)}(z) \mathring{C}_r.\end{gather*}
 Then, there exists a fundamental matrix $Y^{(0)}(z,u)$ of~\eqref{11marzo2018-TRIS} in Levelt form such that the essential monodromy data of~\eqref{11marzo2018-TRIS} on the whole $\mathcal{U}_\epsilon\big(u^C\big)$ are
 \begin{gather*}B,\quad \mathbb{S}_r=\mathring{\mathbb{S}}_r,
 \quad
 \mathbb{S}_{r+1}=\mathring{\mathbb{S}}_{r+1} ,
 \quad
 C_r=\mathring{C}_r
 ,
 \quad
 L=\mathring{L},
 \quad
 D.\end{gather*}

\item[$d)$] If the diagonal entries of $A(u^c)$ do not differ by non-zero integers, then there is a unique formal solution $\mathring{Y}_F(z)$ which necessarily satisfies $\mathring{Y}_F(z)=Y_F\big(z,u^C\big)$.
\end{enumerate}
 \end{Theorem}

\begin{proof} Points a), b), c) follow from Propositions~\ref{23marzo2018-1} and~\ref{23marzo2018-2}, Theorem~\ref{23marzo2018-3} and Corollary~\ref{31luglio2018-1}. We only need to justify that given a fundamental matrix $ \mathring{Y}^{(0)}(z)$ of~\eqref{23marzo2018-6} in Levelt form~\eqref{31luglio2018-7}, there exists a fundamental matrix~$Y^{(0)}(z,u)$ of $\dd Y=\omega(z,u) Y$ in Levelt form (namely an isomonodromic fundamental matrix of \eqref{11marzo2018-TRIS}), such that
 \begin{gather*}Y^{(0)}\big(z,u^C\big)=\mathring{Y}^{(0)}(z).\end{gather*}
To this end, recall the proof of Proposition \ref{30luglio 2018-3}, with the replacements $u\mapsto u^C$ and $\mathbb{D}\big(u^0\big) \mapsto \mathcal{U}_\epsilon\big(u^0\big)$ (this becomes the proof of Corollary \ref{31luglio2018-1}). The gauge transformation \eqref{31luglio2018-3}
\begin{gather*}
Y=V(z,u) \widetilde{Y},\qquad V\big(z,u^C\big)=I,
\end{gather*}
 transforms the Pfaffian system to a differential system \eqref{31luglio2018-4} in variable $z$ only, with $u^0$ replaced by $u^C$, namely
\begin{gather*}
 \dd\widetilde{Y}= \left[ \left(\frac{A\big(u^C\big)}{z}+\Lambda\big(u^C\big)\right)\dd z\right] \widetilde{Y},\qquad \dd \equiv \frac{\partial }{\partial z} \dd z.
 \end{gather*}
 This is exactly \eqref{23marzo2018-6}. A fundamental solution $\mathring{Y}^{(0)}(z)$ in Levelt form \eqref{31luglio2018-7} is actually a solution~\eqref{31luglio2018-5}, namely
 \begin{gather*}
 \mathring{Y}^{(0)}(z)=W(z) z^D z^{\mathring{L}},\qquad W(z)=W_0 +\sum_{\ell=1}^\infty W_\ell z^\ell, \qquad \det(W_0)\neq 0.
 \end{gather*}
 Thus, we conclude from \eqref{31luglio2018-3} that
 \begin{gather*}
 Y^{(0)}(z,u)= V(z,u) \mathring{Y}^{(0)}(z), \qquad Y^{(0)}\big(z,u^C\big)= \mathring{Y}^{(0)}(z).
 \end{gather*}

 Point d) follows from the recursive computation of the coefficients $\mathring{F}_k$ in
 \begin{gather*}
 \mathring{Y}_F(z)=\left(I+ \sum_{k=1}^\infty \mathring{F}_kz^{-k} \right)z^{B(u^C)}e^{z \Lambda(u^C)}.
 \end{gather*}
 For $1\leq \ell \neq s \leq n$ such that $u_\ell^C=u_s^C$, one finds at step $k$ that (recall that $\operatorname{diag}(A(u))=B$ is constant)
 \begin{gather*}
 (A_{\ell\ell}-A_{ss}+k)(\mathring{F}_k)_{\ell s}= \text{known expression from previous steps}.
 \end{gather*}
 This implies that if $A_{\ell\ell}-A_{ss}+k$ is different from zero for every $k\geq 1$, then all the $\mathring{F}_k$ are uniquely recursively determined. For the detailed formulae of the recursion, see \cite[Section~4]{CDG}.
 \end{proof}

In presence of the vanishing conditions \eqref{22marzo2018-1}, Theorem \ref{25marzo2018-1} assures that in order to find the monodromy data of the differential system \eqref{11marzo2018-TRIS} on $\mathcal{U}_\epsilon\big(u^C\big)$, we just need to find the monodromy data associated with the fundamental matrices $\mathring{Y}_r(z)$, $\mathring{Y}_{r+1}(z)$ of \eqref{23marzo2018-6} asymptotic to $ \mathring{Y}_F(z)=Y_F\big(z,u^C\big)$ in $\widehat{\mathcal{S}}_r\big(u^C\big)$, and with a solution $\mathring{Y}^{(0)}(z)= \mathring{\widehat{Y}}(z)z^D z^{\mathring{L}}$ of \eqref{23marzo2018-6} in Levelt form.

The computation of monodromy data of \eqref{11marzo2018-TRIS} (namely \eqref{11marzo2018-1}) is highly transcendental, so that even if $A(u)$ is completely known in a neighbourhood of $u^C$, in general it cannot be done at a~generic $u\in \mathcal{U}_\epsilon\big(u^C\big)$. On the other hand, it may happen that it can be explicitly done for~\eqref{23marzo2018-6}, because the system simplifies at the coalescence point, thanks to the null entries
\begin{gather*}
A_{ij}\big(u^C\big)=0 \qquad \text{whenever}\quad u_i^C=u_j^C.
\end{gather*} By Theorem \ref{25marzo2018-1}, the result so obtained yields the constant monodromy data of \eqref{11marzo2018-TRIS} in a~neighbourhood of~$u^C$. An example of this has been given in \cite{CDG} concerning monodromy data of a~Painlev\'e equation, and in \cite{CDG1} for the monodromy data of the Frobenius manifold associated with the reflection group~$A_3$.

Theorem \ref{25marzo2018-1} allows to compute monodromy data in a neighbourhood of a coalescence point also in cases when {\it we miss some information about the system away form $u^C$}. Indeed, suppose that the system \eqref{11marzo2018-TRIS} {\it is known only at} a coalescence point $u^C$; namely, we only know the explicit form of $A\big(u^C\big)$. Moreover, suppose that for some theoretical reason it is known that in a neighbourhood of $u^C$ the {\it unknown} $A(u)$ must satisfy the vanishing conditions~\eqref{22marzo2018-1}. Then, Theorem~\ref{25marzo2018-1} yields the essential monodromy data in a whole neighbourhood of $u^C$ without knowing explicitly $A(u)$ for $u\neq u^C$. This approach has been used in \cite{CDG1} to compute the monodromy data of the quantum cohomology of the Grassmannian ${\rm Gr}(2,4)$, and to prove in a~completely explicit way (i.e., by computations of the numerical values of the data) a~conjecture \cite{CDG1,C-et-all,Dub3,GGI} relating them to exceptional collections in derived categories of coherent sheaves on ${\rm Gr}(2,4)$ (for a~shorter summary, see also \cite{Eretico2}, which is based on a talk at CRM, Pisa, February 13--17, 2017). Notice that the quantum cohomology of almost all Grassmannians is characterised by a coalescence phenomenon~\cite{Cotti1} and that Theorem~\ref{25marzo2018-1} applies, due to the semisimplicity of these quantum cohomologies. This fact justifies the computation of the monodromy data for Grassmannians starting from a coalescence point.

\begin{Remark} When applying Theorem \ref{25marzo2018-1} as explained above, among the possible formal solutions $\mathring{Y}_F(z)$ of \eqref{23marzo2018-6}, we need precisely the one satisfying $\mathring{Y}_F(z)=Y_F\big(z,u^C\big)$. The fact that we do not have this information without knowing already $Y_F(z,u)$ may constitute a difficulty. Fortunately, the problem does not arise when the diagonal entries of $A\big(u^C\big)$ do not differ by non-zero integers, because in this case \eqref{23marzo2018-6} only has the unique formal solution $\mathring{Y}_F(z)=Y\big(z,u^C\big)$. In applications to Frobenius manifolds, discussed in \cite{CDG,CDG1,C-et-all,Eretico2}, we are in this good situation. Notice that the Stokes matrices are completely determined by \eqref{23marzo2018-6} in this case.\footnote{Since the choice of a solution in Levelt form is not unique, there may be a freedom in the data $C_r$, $D$, $L$. Only the Stokes matrices are uniquely determined.}
\end{Remark}

\begin{Remark}\label{10aprile2018-1} For a system analogous to \eqref{11marzo2018-1}, associated with a semisimple Frobenius manifold, where $A(u)$ is skew symmetric, a synthetic proof is given in \cite{GGI} that a fundamental matrix solution asymptotic to the formal solution in a sector of central opening angle $\pi+\varepsilon$ (the ana\-logous to our $Y_r(z,u)\sim Y_F(z,u)$) is holomorphic in a small neighbourhood of a coalescence point~$u^C$. This result, in case of Frobenius manifolds, is analogous of the first point of Theo\-rem~\ref{25marzo2018-1}. The proof in~\cite{GGI} is based on the Laplace transformation of the irregular system into an isomonodromic Fuchsian system, whose associated Pfaffian system is of Fuchsian type and has the good analyticity properties discussed in~\cite{YT}.
\end{Remark}

\appendix

\section[Weak and strong isomonodromic deformations of Fuchsian systems]{Weak and strong isomonodromic deformations \\ of Fuchsian systems}\label{appendixA}

The difference between weak and strong isomonodromic deformations naturally arises in case of Fuchsian systems. This is synonymous of Schlesinger and non-Schlesinger deformations studied by Bolibruch.

Up to a M\"obius transformation, we can assume that $z=\infty$ is non singular. Accordingly, we consider the $u$-family of $n\times n$ Fuschsian systems
\begin{gather} \label{25febbraio2018-3}
\frac{\dd Y}{\dd z}=\sum_{i=1}^N\frac{A_i(u)}{z-u_i}Y,\qquad \sum_{i=1}^NA_i(u)=0,
\end{gather}
depending holomorphically on the parameter $u=(u_1,\dots,u_N)$ in a polydisc $\mathbb{D}\big(u^0\big)$ with center $u^0=\big(u_1^0,\dots,u_N^0\big)$, contained in $\mathbb{C}^N\backslash \bigcup_{i\neq j}\{(u_i-u_j)=0\}$. We take $(z,u)$ in
 \begin{gather*}
 \mathcal{E}:=\left(\mathbb{P}^1\times \mathbb{D}\big(u^0\big)\right)\backslash \bigcup_{i=1}^N\left\{(z-u_i)=0\right\} .
 \end{gather*}
 Notice that $\mathbb{D}\big(u^0\big)=
\mathbb{D}_1\times \cdots \times \mathbb{D}_N$, where $\mathbb{D}_i$ is a disk centered at $u_i^0$ and $\overline{\mathbb{D}}_i\cap \overline{\mathbb{D}}_j=\emptyset$.
Thus, the fundamental group of $\mathbb{P}^1\backslash\{u_1,\dots,u_N\}$ can be defined in a $u$-independent way.\footnote{In other words, $ \mathcal{E}=\mathbb{P}^1\times \mathbb{D}\big(u^0\big)\backslash \bigcup_{i=1}^N\{(z-u_i)=0\}$ can be retracted to $\mathbb{P}^1\backslash \big\{u_1^0,\dots,u_N^0\big\}$, and $\pi_1\big( \mathcal{E}; \big(z_0,u^0\big)\big)$ is isomorphic to $\pi_1\big(\mathbb{P}^1\backslash \big\{u_1^0,\dots,u_N^0\big\};z_0\big)$.}

\begin{Definition} The family of systems \eqref{25febbraio2018-3} is weakly isomonodromic if there exists a family of fundamental matrix solutions $Y(z,u)$ having the same monodromy matrices for all $u\in \mathbb{D}(u_0)$. This family $Y(z,u)$ is called a weak isomonodromic family of fundamental matrices.
\end{Definition}

It follows from the theorem on analytic dependence on parameters that there exists a~fa\-mi\-ly of fundamental matrices $\widetilde{Y}(z,u)$ which is analytic in $(z,u)$ on the universal covering of~$ \mathcal{E}$, but which is not necessarily isomonodromic. Namely, $\widetilde{Y}(z,u)=Y(z,u)C(u) $ for some connection matrix $C(u)$. If $M_1,\dots,M_N$ are the monodromy matrices of $Y(z,u)$ w.r.t.\ a~basis of $\pi_1\big(\mathbb{P}^1\backslash \big\{u_1^0,\dots,u_N^0\big\};z_0\big)$, then $\widetilde{Y}(z,u)$ has non-constant matrices $M_j(u)=C(u)^{-1}M_j C(u)$. For a Fuchsian isomonodromic system, it can be proved that there exists an isomonodromic $Y(z,u)$ such that $C(u)$ is holomorphic. Namely:

\begin{Proposition}[\cite{Bolibruch,Bolibruch1}]\label{25febrraio2018-5} If \eqref{25febbraio2018-3} is weakly iso\-mo\-no\-dromic, then there exists an isomonodromic family of fundamental matrices $Y(z,u)$ which is holomorphic in $z$ and $u$ on the universal covering of $ \mathcal{E}$.
\end{Proposition}
The proof makes use of holomorphic bundles, and we refer to \cite{Bolibruch,Bolibruch1}.

\begin{Theorem}[\cite{Bolibruch,Bolibruch1}]\label{28febbraio2018-1}
The $u$-family of Fuschsian systems \eqref{25febbraio2018-3} is weakly isomonodromic if and only if there exists a matrix valued $1$-form $\omega$, holomorphic on $ \mathcal{E}$, such that
\begin{itemize}\itemsep=0pt
\item[$1)$] for any fixed $u\in \mathbb{D}\big(u^0\big)$ we have
\begin{gather*}\omega|_{u} = \sum_{i=1}^N\frac{A_i(u)}{z-u_i} \dd z;\end{gather*}
\item[$2)$] $\dd\omega=\omega\wedge \omega$.
\end{itemize}
\end{Theorem}

\begin{proof} Condition 2) is the Frobenius integrability condition for the Pfaffian system $\dd Y=\omega Y$. Condition 1) assures that the $z$ part of the Pfaffian system is the Fuchsian system~(\ref{25febbraio2018-3}). If the deformation is isomonodromic, then by Proposition~\ref{25febrraio2018-5} there is a holomorphic family $Y(z,u)$ and as in the case of~\eqref{25marzo2018-4-MENO}--\eqref{25marzo2018-4}, we see that $\omega=\dd Y Y^{-1}$ is single valued and holomorphic on~$ \mathcal{E}$. This also implies that the Pfaffian system is integrable, then 2) holds. Conversely, if 2) holds, then by the general properties of linear Pfaffian systems there exist a solution~$Y$, holomorphic on the universal covering of $ \mathcal{E}$, whose monodromy -- as in \eqref{25marzo2018-4-MENO}--\eqref{25marzo2018-4} -- is independent of~$u$.
\end{proof}

We are going to show below that for a non-resonant Fuchsian system (i.e., the eigenvalues of the matrices $A_i$ do not differ by non-zero integers) all weak isomonodromic deformations are actually strong, namely not only the monodromy matrices are constant, but also certain essential monodromy data do not depend on~$u$ (see Definition~\ref{26marzo2018-5} below). On the other hand, if the system is resonant, there exist both weak and strong isomonodromic deformations, which are inequivalent. To say it in other well known words, in the non-resonant case only {\it Schlesinger} deformations exist, while in the resonant case also {\it non-Schlesinger} deformations appear. The former are strong (preserving essential data), the latter are weak (preserving only monodromy matrices). First, let us recall Schlesinger deformations.

{\bf Schlesinger deformations.} For a Fuchsian system
\begin{gather}\label{25febbraio2018-9}
\frac{\dd Y}{\dd z}=\sum_{i=1}^N\frac{A_i^0}{z-u_i^0} Y,\qquad \sum_{i=1}^N A_i^0=0,
\end{gather}
it is always possible to consider its Schlesinger deformations, given by $\omega=\omega_s$, where
\begin{gather*}
\omega_s:=\sum_{i=1}^N \frac{A_i(u)}{z-u_i}\dd (z-u_i),\qquad A_i\big(u^0\big)=A_i^0.
\end{gather*}
Such deformations exist if and only if there exist $A_1(u), \dots, A_N(u)$ such that
\begin{gather*}
\dd\omega_s=\omega_s\wedge \omega_s, \qquad A_i\big(u^0\big)=A_i^0, \qquad \sum_{i=1}^N A_i(u)=0.
\end{gather*}
A standard computation shows that the condition $\dd\omega_s=\omega_s\wedge \omega_s$ is equivalent to the {\it Schlesinger equations}
 \begin{gather*}
 \dd A_i(u)= \sum_{j\neq i} [A_j(u),A_i(u)]\frac{\dd (u_j-u_i)}{u_j-u_i}.
 \end{gather*}
 This implies that \begin{gather*}\sum_{j=1}^N \frac{\partial A_i}{\partial u_j}=0,\end{gather*} which ensures that $\sum_i A_i(u)=0$ whenever $\sum_i A_i\big(u^0\big)=0$. The Schlesinger system is well known to be Frobenius integrable. We conclude that $\omega_s$, called {\it Schlesinger deformation} of the Fuchsian system (\ref{25febbraio2018-9}), always exists, including the resonant case (but it is not the unique one in this case). For an isomonodromic holomorphic fundamental matrix $Y_s(z,u)$ of $\omega_s$, we have
 \begin{gather*}
 \dd_u Y_s(\infty,u) Y_s(\infty,u)^{-1}=-\sum_{j=1}^\infty\left. \frac{A_i(u)}{z-u_i}\right|_{z=\infty} \dd u_i=0 ,
 \end{gather*}
 so that $
 Y_s(\infty,u)$ is a constant independent of $u$ (for example $ Y_s(\infty,u)=I$). For this reason, we say that the Schlesinger deformation $\omega_s$ is {\it normalized}.

{\bf Non-normalized Schlesinger deformations \cite{IKSY}.} Together with $\omega_s$, there always exist deformations
\begin{gather}\label{2agosto2018-2}
\omega =\omega_s +\sum_{i=1}^N \gamma_i(u)\dd u_i,
\end{gather}
with holomorphic matrix coefficients $\gamma_i(u)$ on $\mathbb{D}\big(u^0\big)$, which are called non-normalized. To see this, let $Y_s(z,u)$ be such that $Y_s(\infty,z)=I$ (or another constant matrix) and $\dd Y_s=\omega_s Y_s$. Then let
\begin{gather*}
\widehat{Y}(z,u):=\Gamma(u) Y_s(z,u),
\end{gather*}
for a holomorphically invertible matrix $\Gamma(u)$ on $\mathbb{D}\big(u^0\big)$ (so that $\widehat{Y}(\infty,u)=\Gamma(u)$ and the nor\-malization at~$\infty$ is lost). We have
\begin{gather*}
\dd\widehat{Y}=\Gamma \dd Y_s+\dd\Gamma Y_s=\Gamma \omega_s Y_s+\dd\Gamma Y_s\\
\hphantom{\dd\widehat{Y}}{}
=\left(\sum_{i=1}^N\frac{ \Gamma A_i \Gamma^{-1}}{z-u_i} \dd(z-u_i)+\sum_{i=1}^N \frac{\partial \Gamma}{\partial u_i} \Gamma^{-1}\dd u_i \right) \widehat{Y}=: \left(\widehat{\omega}_s+\sum_{i=1}^N \gamma_i(u) \dd u_i\right)\widehat{Y}.
 \end{gather*}
 Conversely, for any $\gamma_1,\dots,\gamma_N$ we need to prove that there exists an holomorphically invertible~$\Gamma$ such that $\dd\Gamma \Gamma^{-1}=\sum_i \gamma_i \dd u_i$, namely
 \begin{gather*}
 \frac{\partial \Gamma}{\partial u_i}=\gamma_i(u)\Gamma,\qquad i=1,\dots,N.
 \end{gather*}
 The integrability condition of this system is
 \begin{gather}\label{25febbraio2018-10}
 \partial_j\gamma_i+\gamma_i \gamma_j=\partial_i\gamma_j+\gamma_j \gamma_i.
\end{gather}
 On the other hand, by assumption, $\omega=\omega_s+\sum_i\gamma_i \dd u_i$ satisfies $\dd\omega=\omega\wedge \omega$. We have
 \begin{gather*}
 \dd\omega=\dd \omega_s +\sum_{ij}\partial_j\gamma_i \dd u_j\wedge \dd u_i =
 \dd \omega_s +\sum_{i<j} (\partial_i\gamma_j-\partial_j \gamma_i) \dd u_i\wedge \dd u_j,
 \end{gather*}
and
\begin{gather*}
\omega\wedge \omega =\omega_s\wedge \omega_s+ \sum_{i<j } (\gamma_i\gamma_j-\gamma_j\gamma_i) \dd u_i\wedge \dd u_j.
\end{gather*}
Using that $\dd \omega_s=\omega_s\wedge \omega_s$, we see that $\dd\omega=\omega\wedge \omega$ implies (\ref{25febbraio2018-10}).

The same considerations as above show that if $\omega$ is a 1-form for an isomonodromic family of Fuchsian systems, then any $\omega + \sum_i \gamma_i(u)\dd u_i$ is again responsible for an isomonodromic deformation.

Schlesinger deformations have the following property.

\begin{Proposition}\label{26febbraio2018-3}Consider the Pfaffian system
\begin{gather}\label{26febbraio2018-2}
\dd Y=\left(\sum_{i=1}^N\frac{A_i(u)}{z-u_i}\dd (z-u_i)+\sum_{i=1}^n \gamma_i(u)\dd u_i\right)Y,
\end{gather}
where the matrices $A_i(u)$ and $\gamma_i(u)$ are holomorphic in $\mathbb{D}\big(u^0\big)$. Then each $A_i(u)$ is holomorphically similar to a constant Jordan form~$J_i$. Namely, there exists~$G_i(u)$ holomorphic in $\mathbb{D}\big(u^0\big)$ such that $J_i=G_i(u)^{-1}A_i(u)G_i(u)$ is Jordan and independent of~$u$. In particular, the spectrum of each $A_i(u)$ is independent of~$u$, so that a Schlesinger isomonodromic deformation is isospectral.
\end{Proposition}

\begin{proof} Suppose we want to prove that $J_1=G_1(u)^{-1}A_1(u)G_1(u)$ is independent of $u$. As far as the local behaviour around $z-u_1=0$ is concerned, (\ref{26febbraio2018-2}) can be rewritten as
\begin{gather} \label{30luglio2018-1}
\dd Y=\left(\frac{P_0(x,u)}{x}\dd x+\sum_{i=1}^N P_i(x,u)\dd u_i\right) Y,
\end{gather}
where the $N+1$ independent variables are $ x:=z-u_1; u_1, u_2,\dots,u_N$. Each $P_k(x,u)$, for $k=0,1,\dots,N$, is holomorphic in a~neighbourhood of $x=z-u_1=0$. Explicitly
\begin{gather}
\frac{P_0}{x}= \frac{A_1(u)}{x}+\sum_{j=2}^N\frac{A_j(u)}{x-(u_j-u_1)},\nonumber\\
\sum_{i=1}^NP_i \dd u_i= \sum_{j=2}^N\frac{A_j(u)}{x-(u_j-u_1)}(\dd u_1-\dd u_j)+\sum_{j=1}^N\gamma_j(u)\dd u_j.\label{30luglio2018-2}
\end{gather}
 So, the proof repeats the arguments of the proofs of Lemma~\ref{24marzo2018-2} and Proposition~\ref{11marzo2018-7}. \end{proof}

We now go back to the most general $\omega$ in Theorem \ref{28febbraio2018-1}, which we are going to characterise. Let $Y(z,u)$ be a holomorphic weak isomonodromic fundamental matrix solution of $\dd Y=\omega Y$. Being in particular a solution of the Fuchsian system $\partial_z Y=\sum_i A_i(u)/(z-u_i)Y$, we write its Levelt form at $z=u_i$
\begin{gather}\label{28febbraio2018-5}
 Y(z,u)= \widehat{Y}_i(z,u) (z-u_i)^{D_i}(z-u_i)^{L_i(u)} C_i(u),
 \end{gather}
 where $C_i(u)$ is an invertible connection matrix, and
 \begin{itemize}\itemsep=0pt
\item[--] $\widehat{Y}_i(z,u)$ is holomorphically invertible at $z=u_i$, and $\widehat{Y}_i(z=u_i,u)=G_i(u)$ of Proposition~\ref{26febbraio2018-3}.
\item[--] $L_i(u)$ is block-diagonal $L_i=L_1^{(i)}\oplus\cdots\oplus L_{\ell_i}^{(i)}$, with upper triangular matrices $L_q^{(i)}$; each $L_q^{(i)}$ has only one eigenvalue $\sigma_q^{(i)}$ satisfying $0\leq \operatorname{Re} \sigma_q^{(i)} <1$, and $\sigma_p^{(i)}\neq \sigma_q^{(i)}$ for $1\leq p\neq q\leq \ell_i$.
\item[--] $D_i$ is a diagonal matrix of integers, which can be split into blocks ${D}_1^{(i)}\oplus\cdots\oplus
{D}_{\ell_i}^{(i)}$. The integers ${d}_{q,r}^{(i)}$ in each ${D}_q^{(i)}=\operatorname{diag}\big({d}_{q,1}^{(i)},{d}_{q,2}^{(i)},\dots\big)$ form a non-increasing (finite) sequence ${d}_{q,1}^{(i)}\geq {d}_{q,2}^{(i)}\geq \cdots$.

\item[--] The eigenvalues of $A_i$ are $\lambda_{q,s}^{(i)}:=\sigma_q^{(i)}+d_{q,s}^{(i)}$. Each block $L_q^{(i)}(u)$ corresponds to the eigenvalues of $A_i(u)$ which differ by non-zero integers.

\item[--] The following holds
\begin{gather*}
\lim_{z\to u_i} \big[ \widehat{Y}_i(z,u) \big(D_i +(z-u_i)^{D_i}L_i(u)(z-u_i)^{-D_i}\big) \widehat{Y}_i^{-1}(z,u)\big]=A_i(u).\end{gather*}
\end{itemize}

\begin{Lemma}[{\cite[Lemma 2]{Bolibruch1}}] $L_i(u)$ and $C_i(u)$ can be taken holomorphic on $\mathbb{D}\big(u^0\big)$, so that also each $\widehat{Y}_i(z,u)$ depends holomorphically on~$u$.
\end{Lemma}

The following lemma shows that a general deformation of Theorem \ref{28febbraio2018-1} is isospectral.

\begin{Lemma}\label{29marzo2018-4} For a weak isomonodromic deformation of a Fuchsian system, the diagonal of~$L_i(u)$ is constant and~$D_i$ is constant, so that the eigenvalues of each~$A_i(u)$ do not depend on~$u$.
\end{Lemma}

\begin{proof} We rewrite \eqref{28febbraio2018-5} as
\begin{gather*}
 Y(z,u)= \widehat{Y}_i(z,u) (z-u_i)^{D_i}C_i(u)(z-u_i)^{\mathcal{L}_i},\\
 \mathcal{L}_i:=C_i(u)^{-1}L_i(u)C_i(u)\equiv \frac{1}{2\pi i } \log M_i, \qquad 0\leq \operatorname{Re} (\text{Eigenvls of }\mathcal{L}_i)<1.
 \end{gather*}
 Since $\dd M_i=0$, then $\dd\mathcal{L}_i=0$. Hence, the eigenvalues $\sigma_q^{(i)}$ of $L_i(u)$ are constant. Now, the eigenvalues $\lambda_{q,s}^{(i)}(u)=\sigma_q^{(i)}+d_{q,s}^{(i)}$ of $A_i(u)$ are continuous, being $A_i(u)$ continuous, so the integers~$d_{q,s}^{(i)}$ are constant.
 \end{proof}

We observe that if $C_i$ and $L_i$ are constant, then $\omega=\omega_s+\sum_i\gamma_i\dd u^i$, as in \eqref{2agosto2018-2}. Indeed, in this case we have
\begin{gather*}
\dd Y Y^{-1}=\dd\widehat{Y}_i \widehat{Y}_i^{-1}+ \widehat{Y}_i \frac{D_i +(z-u_i)^{D_i}L_i(z-u_i)^{-D_i}}{z-u_i} \widehat{Y}_i^{-1}\dd (z-u_i)\\
\hphantom{\dd Y Y^{-1}}{}= \widehat{Y}_i(u_i) \frac{D_i +\lim\limits_{z\to u_i}(z-u_i)^{D_i}L_i(z-u_i)^{-D_i}}{z-u_i} \widehat{Y}_i^{-1}(u_i)\dd(z-u_i) + \text{holomorphic} \\
\hphantom{\dd Y Y^{-1}}{} =\frac{A_i(u)}{z-u_i}\dd (z-u_i)+ \text{holomorphic}.
 \end{gather*}
The holomorphic part is regular at $z=u_i$, and moreover $\dd YY^{-1}=\dd \Gamma \Gamma^{-1}$ at $z=\infty$, having fixed $\Gamma(u):=Y(\infty,u)$. Thus, Liouville theorem yields $\omega=\omega_s+\dd\Gamma \Gamma^{-1}$, which is~\eqref{2agosto2018-2}. The converse of the above is also true:

\begin{Lemma} [\cite{Bolibruch0,Bolibruch1,KV,YT}]If the deformation is Schlesinger, i.e., $\omega=\omega_s$, or its non-normalized version~\eqref{2agosto2018-2}, then $C_i$ and $L_i$ are constant.
\end{Lemma}

\begin{proof} We prove that $L_1$ and $C_1$ are constant, being the other cases analogous. The Pfaffian system is rewritten in the form \eqref{30luglio2018-1}--\eqref{30luglio2018-2}. Then, by \cite[Theorem~5]{YT}, there exists a gauge transformation
\begin{gather*}
Y=V(x,u)\widetilde{Y},
\end{gather*}
holomorphic in a neighbourhood of $(x,u)=(0,u^0)$, given by a convergent series
\begin{gather*}
V(x,u)=I+\sum_{k_1+\cdots +k_N>0} V_k x^{k_0}\big(u_1-u_1^0\big)^{k_1}\cdots\big(u_N-u_N^0\big)^{k_N} ,
\end{gather*}
with $k=(k_0,k_1,\dots,k_N)$, such that the Pfaffian system is transformed to
\begin{gather*}
\dd \widetilde{Y}=
\left(\frac{P_0(0,u^0)}{x}\dd x\right) \widetilde{Y} \equiv \left[ \left(\frac{A_1\big(u^0\big)}{x}+\sum_{j=2}^N\frac{A_j\big(u^0\big)}{x-\big(u_j^0-u_1^0\big)}\right) \dd x\right] \widetilde{Y}.
\end{gather*}
Thus, $\widetilde{Y}$ satisfies an $n\times n$ Fuchsian system in variable $x$, independent of parameters (being $u^0$ fixed), so that there is a fundamental matrix solutions in Levelt form at $x=0$, independent of~$u$,
\begin{gather*}
\widetilde{Y}_1(x):=W(x) x^{D_1} x^{L_1} ,\qquad W(x) = W_0+\sum_{\ell=1}^\infty W_\ell x^\ell,
\end{gather*}
the series being convergent and holomorphic at $x=0$. Hence, there is a fundamental matrix solution of the initial system of the form
\begin{gather*}
Y_1(z,u) =V(x,u) W(x) x^{D_1} x^{L_1} \equiv \widehat{Y}_1(z,u) (z-u_1)^{D_1}(z-u_1)^{L_1},
\end{gather*}
where
\begin{gather*}
\widehat{Y}_1(z,u):= V(x,u) W(x) |_{x=z-u_1}\\
= \left(I+\sum_{k_1+\cdots +k_N>0} V_k (z-u_1)^{k_0}\big(u_1-u_1^0\big)^{k_1}\cdots\big(u_N-u_N^0\big)^{k_N}\right) \left(W_0+\sum_{\ell=1}^\infty W_\ell (z-u_1)^\ell\right).
\end{gather*}
The latter can be rewritten as a convergent saries at $z=u_1$ with coefficients holomorphic in $u$:
\begin{gather*}
\widehat{Y}_1(z,u)=G_1(u) \left(I+\sum_{\ell=1}^\infty \Psi_\ell(u)(z-u_1)^\ell\right),\qquad G_1(u)=V(0,u)W_0.
\end{gather*}
Notice that $G_1(u)$ puts $A_1(u)$ in Jordan form as in Proposition~\ref{26febbraio2018-3}. The proof is concluded observing that $Y_1(z,u)$ solves the initial Pfaffian system, so that also $Y(z,u)=Y_1(z,u)C_1$ is a~fundamental solution if and only if $C_1$ is constant. The proof can also be deduced directly from \cite[Theorems~4 and~7]{YT}.
\end{proof}

From the above, we have the following
\begin{Proposition}\label{28febbraio2018-6}The deformation in Theorem~{\rm \ref{28febbraio2018-1}} is Schlesinger, namely $\omega=\omega_s$, or its non-normalized version~\eqref{2agosto2018-2}, if and only if the Pfaffian system admits an isomonodromic fundamental matrix \eqref{28febbraio2018-5} such that $L_i$ and $C_i$ are constant.
\end{Proposition}

\begin{Definition}\label{26marzo2018-5}The {\rm essential monodromy data} associated with a weakly isomonodromic $Y(z{,}u)$ are the matrices
\begin{gather*}
D_j, \ L_j, \ C_j,\qquad j=1,\dots,N.
\end{gather*}
\end{Definition}

\begin{Definition} We call $Y(z,u)$ as {\rm strongly isomonodromic} fundamental matrix if the essential monodromy data are constant on $\mathbb{D}\big(u^0\big)$. Equivalently, each $Y_j(z,u):=\widehat{Y}_j(z,u) (z-u_j)^{D_j}(z-u_j)^{L_j}$ is a fundamental matrix of $\dd Y=\omega Y$.
\end{Definition}

\begin{Remark}Recall that the $D_j$ are always constant for a weak isomonodromic deformation (Lemma~\ref{29marzo2018-4}). Moreover, \begin{gather*}\dd C_j=0 \quad \text{implies} \quad \dd L_j=0,
\end{gather*}
because we can write the monodromy matrices as $M_j=C_j^{-1}e^{2\pi i L_j}C_j$. Thus, one can define a~strong deformation to be one such that all the matrices $C_j$ are constant, in the same way as in Definition \ref{18marzo2018-3} (the connection matrices $C_j$ play the same role of the matrices $H_r$).
\end{Remark}

Thus, we conclude that
\begin{Proposition}[Proposition \ref{28febbraio2018-6} revised]
An isomonodromic deformation is strong if and only if it is a Schlesinger deformation or its non-normalized version \eqref{2agosto2018-2}.
\end{Proposition}

Following \cite{KV}, we may call $\mathcal{P}_i(z,u):=(z-u_i)^{D_i}(z-u_i)^{L_i(u)}C_i(u)$ the {\it principal part} of $Y(z,u)$ at $z=u_i$, so that
\begin{gather*}
Y(z,u)=\widehat{Y}_i(z,u) \mathcal{P}_i(z,u).
\end{gather*}
In case $C_i$ and $L_i$ are constant, then
 \begin{gather*}\mathcal{P}_i(z,u)\equiv \mathcal{P}_i(z-u_i),\end{gather*}
namely, $\mathcal{P}_i$ depends on $(z,u)$ only through the combination $(z-u_i)$. If such a dependence occurs at any $u_i$, the deformation is called {\it isoprincipal} in~\cite{KV}. The main result of \cite{KV} is that a~family of Fuchsian systems (\ref{25febbraio2018-3}) is an isoprincipal deformation if and only if the matrices $A_i$ satisfy the Schlesinger equations. This result is equivalent to Proposition~\ref{28febbraio2018-6}.

We are ready to review the most general form of $\omega$.

\begin{Theorem}[non-Schesinger deformations, Bolibruch \cite{Bolibruch,Bolibruch1}] \label{25marzo2018-7}
All the possible matrix differential $1$-form of Theorem~{\rm \ref{28febbraio2018-1}}, holomorphic on $\mathbb{P}^1 \times \mathbb{D}\big(u^0\big)\backslash \bigcup_{i=1}^n \{z-u_i=0\}$, which give an isomonodromic family of Fuchsian systems \eqref{25febbraio2018-3}, have the form
\begin{gather}\label{1marzo2018-1}
\omega =\omega_s +\sum_{i=1}^N \gamma_i(u) \dd u_i + \sum_{i=1}^N\sum_{j=1}^N\sum_{k=1}^{m_i} \frac{ \gamma_{ijk}(u)}{(z-u_i)^k}\dd u_j.
\end{gather}
where $\gamma_i(u)$, $\gamma_{ijk}(u)$ are holomorphic on $\mathbb{D}\big(u^0\big)$ and $m_i $ is the maximal integer difference of eigenvalues of~$A_i(u)$.
\end{Theorem}

\begin{proof} The proof is in \cite{Bolibruch} and \cite{Bolibruch1}. We take the Schlesinger deformation $\omega_s$ with the same initial condition $A_i^0=A_i\big(u^0\big)$ of $\omega$, namely
\begin{gather*}
\omega_s|_{u^0}=\omega|_{u^0}.
\end{gather*}
Let $Y_s(z,u)$ be an isomonodromic fundamental matrix for $\omega_s$. By isomonodromicity, $Y(z,u)$ and $Y_s(z,u)$ have the same monodromy matrices, which are those of
\begin{gather*}\frac{\dd Y}{\dd z}=\sum\limits_{i=1}^N\frac{A_i^0}{z-u_i^0} Y.
\end{gather*}
Hence, $\Gamma(z,u):=Y(z,u) Y_s(z,u)^{-1}$ is single valued and meromorphic on $\mathbb{P}^1\times \mathbb{D}\big(u^0\big)$, with poles at $z=u_i$, $1\leq i \leq N$. Moreover $\omega = \dd\Gamma \Gamma^{-1} +\Gamma\omega_s \Gamma^{-1}$. This expression does show that $\omega$ has the structure~(\ref{1marzo2018-1}). In order to predict~$m_i$, we write
\begin{align*}
&Y(z,u)=\widehat{Y}_i(z,u) (z-u_i)^{D_i}(z-u_i)^{L_i(u)}C_i(u),
\\
&
Y_s(z,u)=\widehat{Y}_i^s(z,u) (z-u_i)^{D_i}(z-u_i)^{L_i^s}C_i^s,
\end{align*}
where $L_i^s$ and $C_i^s$ are constant, by Proposition \ref{28febbraio2018-6}. Being the monodromy $M_i$ of $Y$ and $Y_s$ the same, we have
\begin{gather*}
C_i(u)^{-1}L_i(u)C_i(u)={C_i^s}^{-1}L_i^sC_i^s.
\end{gather*}
 Thus,
 \begin{gather*}
 L_i(u)=\mathcal{R}_i L_i^s \mathcal{R}_i^{-1},
 \qquad
 \mathcal{R}_i:= C_i(u){C_i^s}^{-1},
 \end{gather*}
 and $\mathcal{R}_i$ must have the same block structure of $L_i(u)$ and $L_i^s$. Moreover,
 \begin{gather*}
 Y(z,u)=\widehat{Y}_i(z,u)(z-u_i)^{D_i} \mathcal{R}_i (z-u_i)^{L_i^s}C_i^s.
 \end{gather*}
 Thus, we have
\begin{gather*}
\dd Y Y^{-1} = \dd\widehat{Y}_i \widehat{Y}_i^{-1} +\widehat{Y}_i \left(\frac{D_i+(z-u_i)^{D_i}\mathcal{R}_i E_i^s \mathcal{R}_i^{-1}(z-u_i)^{-D_i}}{z-u_i}\right) \widehat{Y}_i^{-1} \dd(z-u_i)\\
\hphantom{\dd Y Y^{-1} =}{} +\widehat{Y}_i(z-u_i)^{D_i}\dd\mathcal{R}_i \mathcal{R}_i^{-1} (z-u_i)^{-D_i} \widehat{Y}_i^{-1}\\
\hphantom{\dd Y Y^{-1}}{} = \dd\widehat{Y}_i \widehat{Y}_i^{-1} +\widehat{Y}_i\left(\frac{D_i+(z-u_i)^{D_i}L_i(u) (z-u_i)^{-D_i}}{z-u_i}\right)\widehat{Y}_i^{-1}\dd(z-u_i)\\
\hphantom{\dd Y Y^{-1} =}{} +\widehat{Y}_i(z-u_i)^{D_i}\dd \mathcal{R}_i \mathcal{R}_i^{-1} (z-u_i)^{-D_i} \widehat{Y}_i^{-1}\\
\hphantom{\dd Y Y^{-1}}{} = \text{holomorphic}+ \frac{A_i(u)}{z-u_i} \dd(z-u_i) +\widehat{Y}_i (z-u_i)^{D_i}\dd\mathcal{R}_i \mathcal{R}_i^{-1} (z-u_i)^{-D_i} \widehat{Y}_i^{-1}.
\end{gather*}
The last term $\widehat{Y}_i (z{-}u_i)^{D_i}\dd\mathcal{R}_i \mathcal{R}_i^{-1} (z{-}u_i)^{-D_i} \widehat{Y}_i^{-1}$ contains matrix entries with poles at \smash{$(z{-}u_i){=}0$}, plus holomorphic terms. This proves again, by Liouville theorem, that $\omega$ has the struc\-tu\-re~(\ref{1marzo2018-1}). The block diagonal structure of $\mathcal{R}_i$, being the same as that of $D_i=D_1^{(i)}\oplus \cdots \oplus D_{\ell_i}^{(i)}$, assures that $m_i$ is the maximum over $q$ of the maximal difference of the eigenvalues of $D_q^{(i)}$, $q=1,2,\dots,\ell_i$. This is precisely the maximal integer difference of eigenvalues of $A_i(u)$.
 \end{proof}

In the non-resonant case, $m_i=0$ for every $i=1,\dots,N$. Therefore, we obtain the following
 \begin{Corollary} If for every $i=1,\dots,N$ the matrices $A_i$ are non-resonant, then every $\omega$ of Theorem~{\rm \ref{28febbraio2018-1}}, defining an isomonodromic family of Fuchsian systems~\eqref{25febbraio2018-3}, is of the form
 \begin{gather*}
 \omega=\omega_s+\sum_{i=1}^n \gamma_i(u)\dd u_i.
 \end{gather*}
 \end{Corollary}
 In other words, the isomonodromy deformation of a non-resonant Fuchsian system are only the Schlessinger deformations, possibly non-normalized. The term $\sum\limits_{i=1}^N\sum\limits_{j=1}^N\sum\limits_{k=1}^{m_i} \frac{ \gamma_{ijk}(u)}{(z-u_i)^k}\dd u_j$ can only appear in the resonant case. Summarizing, we have found that
 \begin{itemize}\itemsep=0pt
 \item[1)] in the non-resonant case, the definitions of weak and strong isomonodromic deformation are equivalent, and the Schlesinger equations are necessary and sufficient conditions;
 \item[2)] in the resonant case, there exist both weak isomonodromic deformations, determined by the form $\omega$ in Theorem~\ref{25marzo2018-7} and called non-Schlesinger, and strong isomonodromic deformations, which are Schlesinger deformations determined by $\omega_s$ (or its non-normalized version).
\end{itemize}

There are examples of non-Schlesinger deformations in \cite{BF, Bolibruch1,Kitaev}.\footnote{One can also have a look at \cite{Po}, though the question there is concerned with Schlesinger transformations that also shift the eigenvalues.} A simple example is given at the end of~\cite{KV}, and we report it here. Consider the family of differential systems \begin{gather*}
 \frac{\dd Y}{\dd z}
 =
 \left(
 \frac{1}{z-u}\left(
\begin{matrix}
1& 0
\\
0 & 0
\end{matrix}
\right)
+\frac{1}{z-1}
\left(
\begin{matrix}
-1 & \dfrac{-u(u-1)h(u)}{u-3}
\\
0 & 0
\end{matrix}
\right)
+\frac{1}{z-2}
\left(
\begin{matrix}
0 & \dfrac{2u(u-2)h(u)}{u-3}
\\
0 & 1
\end{matrix}
\right)
\right.
\\
\left.\hphantom{\frac{\dd Y}{\dd z}=}{}
+\frac{1}{z-3}
\left(
\begin{matrix}
0& -uh(u)
\\
0 & -1
\end{matrix}
\right)
 \right)Y.
 \end{gather*}
 A fundamental matrix solution is
 \begin{gather*}Y(z,u) =
\left(
\begin{matrix}
\dfrac{z-u}{z-1} & \dfrac{-2t(z-u)h(u)}{(z-1)(z-3)(u-3)}\\
0 & \dfrac{z-2}{z-3}
\end{matrix}
\right).
 \end{gather*}
 Here $h(u)$ is an arbitrary function, which is holomorphic at $u=0$. The system is resonant at all the Fuchsian singularities $z=u,1,2,3$. One can check that the residue matrices do not satisfy the Schlesinger equations, but the $u$-deformation is isomonodromic in the weak sense, because~$Y(z,u)$ is single-valued and hence has trivial monodromy matrices $M_i=I:=\operatorname{diag}(1,1)$. For $z$ close to $u$, we have
\begin{gather*}
Y(z,u)=\left(
\begin{matrix}
\dfrac{1}{u-1} & 0
\\
0 & \dfrac{u-2}{u-3}
\end{matrix}
\right) (I+O(z-u)) \left(
\begin{matrix}
z-u & 0
\\
0 & 1
\end{matrix}
\right),
\end{gather*}
with {\it constant} connection matrix $=I$. This is not the case at other poles of the system. For example, for $z$ close to 1, we have
\begin{gather*}
Y(z,u)= \left(
\begin{matrix}
\dfrac{u(1-u)h(u)}{u-3} & 1
\\
1 & 0
\end{matrix}
\right) (I+O(z-1)) \left(
\begin{matrix}
1& 0
\\
0 & (z-1)^{-1}
\end{matrix}
\right) \left(
\begin{matrix}
0 & 1/2
\\
1-u & \dfrac{u(1-u)h(u)}{u-3}
\end{matrix}
\right),
\end{gather*}
with $u$-dependent connection matrix
\begin{gather*}
C_1(u)=\left(
\begin{matrix}
0 & 1/2
\\
1-u & \dfrac{u(1-u)h(u)}{u-3}
\end{matrix}
\right).
\end{gather*}
So, the deformation is not isomonodromic in the strong -- or isoprincipal -- sense. The monodromy matrix is nevertheless constant and equal to $I$.

As noted in \cite{KV}, the example shows that a weak isomonodromic deformation may not have the Painlev\'e property. Indeed, in the above example we can choose~$h(u)$ arbitrarily (provided it is holomorphic at $u=0$).
On the other hand, if the deformation is isomonodromic in the strong sense, namely Schlesinger, then the Painlev\'e property holds also in the resonant case.

\section{Proof of Proposition \ref{16marzo2018-1}}\label{appendixB}

We recall that a small polydisc $\mathbb{D}^\prime$ has been introduced by Sibuya to guarantee that the fundamental matrices $Y_r(z,u)$ in~\eqref{11marzo2018-10-BIS} are holomorphic of~$u$.

\begin{Lemma}\label{16marzo2018-2} Let $\omega_0(x)=\big(\Lambda+\frac{A}{z}\big)$ be as in~\eqref{11marzo2018-6}, let $\omega_1,\dots,\omega_n$ be holomorphic of $(z,u)\in \mathbb{C}\times \mathbb{D}\big(u^0\big)$, and~$A(u)$ holomorphic on $\mathbb{D}\big(u^0\big)$. Assume that $z=\infty$ is at most a pole of $\omega_1,\dots,\omega_n$. Then, for $u\in \mathbb{D}^\prime$, the Pfaffian system has the following structure
\begin{gather*}
\omega=\left(\Lambda+\frac{A}{z}\right)\dd z +\sum_{j=1}^n ( zE_j+\omega_j(0,u)) \dd u_j ,
 \end{gather*}
 with
\begin{gather*} 
 \omega_j(0,u)=[F_1(u),E_j]+\mathcal{D}_j(u),
\end{gather*}
 where $F_1$, appearing in the formal expansion \eqref{14marzo2018-1}, is given explicitly in \eqref{18marzo2018-8}, and
\begin{gather*}
 \mathcal{D}_j(u)=\frac{\partial H_r(u)}{\partial u_j}H_r(u)^{-1}
\end{gather*}
 is a diagonal matrix independent of $r\in\mathbb{Z}$ and holomorphic on $u\in \mathbb{D}^\prime$. Moreover,
 \begin{gather*}\dd B=0. \end{gather*}
\end{Lemma}

\begin{proof} We write
\begin{gather*}
\omega =\dd Y Y^{-1}.
\end{gather*}
where $Y(z,u)=Y^{(0)}(z,u) H$ is a holomorphic isomonodromic family. Recalling that $D$ is constant and $\dd L=\dd H=0$ (one can take directly $H=I$), we have
\begin{gather*}
\omega_j=\frac{\partial Y}{\partial u_j} Y^{-1} =\frac{\partial \widehat{Y}^{(0)}}{\partial u_j} \big(\widehat{Y}^{(0)}\big)^{-1}= \operatorname{holom}(z,u),
 \end{gather*}
where $\operatorname{holom}(z,u)$ is some convergent Taylor series at $z=0$, of order $\mathcal{O}(1)$ for $z\to 0$, with holomorphic coefficients depending on $u\in\mathbb{D}\big(u^0\big)$. Keeping into account the explicit form \eqref{3aprile2018-1} of $\widehat{Y}^{(0)}$ (in particular, recall that $G(u)$ satisfies \eqref{23marzo2018-7} as in Proposition \ref{11marzo2018-7}), we have
\begin{gather}
\label{14marzo2018-8}
\omega_j=\partial_jG G^{-1}+\mathcal{O}(z) \equiv \omega_j(0,u)+\mathcal{O}(z).
\end{gather}

Successively, we compute the representation of $\omega_j$ close to $z=\infty$, using \eqref{14marzo2018-4},
\begin{align*}
\omega_j&=\frac{\partial Y}{\partial u_j} Y^{-1}\\
&=\frac{\partial \widehat{Y}_r}{\partial u_j} \widehat{Y}_r^{-1}+\widehat{Y}_r\frac{\partial B}{\partial u_j} \widehat{Y}_r^{-1} \log z + z\widehat{Y}_rE_j\widehat{Y}_r^{-1}
+\widehat{Y}_rz^B\left(e^{z\Lambda}\frac{\partial H_r}{\partial u_j}H_r^{-1} e^{-z\Lambda}\right)z^{-B}\widehat{Y}_r^{-1}.
\end{align*}
We recall that $u_i\neq u_j$ for $i\neq j$ and observe that $e^{(u_i-u_j)z}$ diverges for $z\to\infty$ in some subsector of $\mathcal{S}_r(\overline{ \mathbb{D}^\prime})$, because the central angular opening is more than~$\pi$. Hence
\begin{gather*}
e^{z\Lambda}\frac{\partial H_r}{\partial u_j}H_r^{-1} e^{-z\Lambda} \quad\text{does not diverge exponentially}
\quad \Longleftrightarrow \quad \frac{\partial H_r}{\partial u_j}H_r^{-1} \quad \text{is diagonal}.\end{gather*}
 Since $\Lambda$ and $B$ are also diagonal, we have
\begin{gather*}
\omega_j=\left(\frac{\partial \widehat{Y}_r}{\partial u_j} \widehat{Y}_r^{-1}+z\widehat{Y}_r E_j \widehat{Y}_r^{-1}+\widehat{Y}_r\frac{\partial H_r}{\partial u_j}H_r^{-1} \widehat{Y}_r^{-1}\right)
+\widehat{Y}_r\frac{\partial B}{\partial u_j} \widehat{Y}_r^{-1} \log z.
\end{gather*}
We have at most a pole at $z=\infty$ if and only if
\begin{gather*}
\frac{\partial B}{\partial u_j}=0.
\end{gather*}
Then, the asymptotic behaviour \eqref{14marzo2018-1} of $\widehat{Y}_r$ implies that
\begin{gather*}
\omega_j=z\widehat{Y}_r E_j \widehat{Y}_r^{-1}
+\widehat{Y}_r\frac{\partial H_r}{\partial u_j}H_r^{-1} \widehat{Y}_r^{-1}
+\frac{\partial \widehat{Y}_r}{\partial u_j} \widehat{Y}_r^{-1}=
zE_j+[F_1,E_j]+\frac{\partial H_r}{\partial u_j}H_r^{-1}+\mathcal{O}\left(\frac{1}{z}\right)
\end{gather*}
in the sector $\mathcal{S}_r(\overline{ \mathbb{D}^\prime})$. The asymptotics holds for every $r\in\mathbb{Z}$. Now, the asymptotic expansion of a function (in our case $\omega_j$) on a sector is unique. Since $\mathcal{S}_r(\overline{ \mathbb{D}^\prime})\cap \mathcal{S}_{r+1} (\overline{\mathbb{D}^\prime})$ is not empty, this implies that $\omega_j$ has the same expansion in every sector $\mathcal{S}_r(\overline{ \mathbb{D}^\prime})$, $\forall\, r$. In particular, $\partial_jH_r H_r^{-1}$ cannot depend on $r$. We set
\begin{gather*}
 \mathcal{D}_j(u):=\frac{\partial H_r(u)}{\partial u_j}H_r^{-1}(u) \qquad \forall\, r\in\mathbb{Z}.
 \end{gather*}
Moreover, being $\omega_j$ single valued, the asymptotic series represented by $\mathcal{O}\big(\frac{1}{z}\big)$ above, must be a~convergent Taylor expansion. In conclusion
\begin{gather*}
\omega_j-(zE_j+[F_1,E_j]+\mathcal{D}_j)=\mathcal{O}\left(\frac{1}{z}\right) \qquad \text{is holomorphic at $z=\infty$}.
\end{gather*}
Keeping into account \eqref{14marzo2018-8} and Liouville theorem, we obtain that
\begin{gather*}
\omega_j-(zE_j+[F_1,E_j]+\mathcal{D}_j)=0,
\end{gather*}
as we wanted to prove.
\end{proof}

\begin{Lemma}\label{16marzo2018-3}Consider a Pfaffian system $\dd Y=\omega Y$, and assume that $\omega$ has the following structure
\begin{gather*}\omega=\left(\Lambda+\frac{A(u)}{z}\right)\dd z+\sum_{j=1}^n ( zE_j+\omega_j(0,u) )\dd u_j,
\end{gather*}
 where $A(u)$ and $\omega_1(0,u),\dots,\omega_n(0,u)$ are holomorphic on $\mathbb{D}\big(u^0\big)$. Then, the Frobenius integrability conditions \eqref{11marzo2018-3} are equivalent to $1)$ and $2)$ below:
 \begin{itemize}\itemsep=0pt
 \item[$1)$] $\omega_j(0,u)$ satisfies the constraint
 \begin{gather} \label{15marzo2018-1}
 [\Lambda,\omega_j(0,u)]=[E_j,A];
 \end{gather}
 \item[$2)$] $A(u)$ satisfies the non-linear system
\begin{gather} \label{15marzo2018-2}
\frac{\partial A}{\partial u_j}=[\omega_j(0,u),A],\qquad j=1,2,\dots,n,
 \end{gather}
which is Frobenius integrable.
\end{itemize}
The above \eqref{15marzo2018-1} determines $\omega_j(0,u)$ up to a diagonal matrix $\mathcal{D}_j(u)$, as follows
\begin{gather}\label{15marzo2018-3}
\omega_j(0,u)=\left( \frac{A_{ab}(u)(\delta_{aj}-\delta_{bj})}{u_a-u_b} \right)_{a,b=1}^n+\mathcal{D}_j(u).
\end{gather}
Moreover, the diagonal matrix $B:=\operatorname{diag}(A_{11},\dots,A_{nn})$ is constant and
 \begin{gather*}
 \mathcal{D}_j=\frac{\partial \mathcal{D}}{\partial u_j},
 \end{gather*}
 where $\mathcal{D}$ is a matrix whose diagonal only depends on $u\in \mathbb{D}\big(u^0\big)$.
\end{Lemma}

\begin{proof} For brevity, let us write $\omega_j(0)$ in place of $\omega_j(0,u)$. As in the proof of Lemma \ref{24marzo2018-2}, consider \eqref{11marzo2018-3} for $\beta=0$, $\alpha=j=1,\dots,n$:
\begin{gather*}
\frac{\partial}{\partial u_j}\left(\Lambda+\frac{A}{z}\right)+\left(\Lambda+\frac{A}{z}\right)(zE_j+\omega_j(0))=
\frac{\partial}{\partial z}(zE_j+\omega_j(0))+(zE_j+\omega_j(0))\left(\Lambda+\frac{A}{z}\right).
\end{gather*}
Since $\partial_j\Lambda=E_j$, we have
\begin{gather*}
\cancel{E_j}+\frac{\partial_jA}{z}+\cancel{z\Lambda E_j}+\Lambda\omega_j(0)+AE_j+\frac{A\omega_j(0)}{z}=\cancel{E_j}+\cancel{zE_j\Lambda} +E_jA+\omega_j(0)\Lambda+\frac{\omega_j(0)A}{z}.
\end{gather*}
The above identity is equivalent to \eqref{15marzo2018-1} and \eqref{15marzo2018-2}. Consider now \eqref{11marzo2018-3} for $\alpha=k$, $\beta=j\in\{1,\dots,n\}$:
\begin{gather*}
\frac{\partial}{\partial u_k}(zE_j+\omega_j(0))+(zE_j+\omega_j(0))(zE_k+\omega_k(0))\\
\qquad{} =\frac{\partial}{\partial u_j}(zE_k+\omega_k(0))+(zE_k+\omega_k(0))(zE_j+\omega_j(0)),
\end{gather*}
so that
\begin{gather*}
\frac{\partial \omega_j(0)}{\partial u_k}+\cancel{z^2E_jE_k}+z(E_j\omega_k(0)+\omega_j(0)E_k)+\omega_j(0)\omega_k(0)\\
\qquad{} =\frac{\partial \omega_k(0)}{\partial u_j}+\cancel{z^2E_kE_j}+z(E_k\omega_j(0)+\omega_k(0)E_j)+\omega_k(0)\omega_j(0).
\end{gather*}
Identifying the coefficients of the same powers of $z$ we get
\begin{gather}\label{26marzo2018-1}
[E_j,\omega_k(0)]=[E_k,\omega_j(0)],\\
\label{15marzo2018-8}
\frac{\partial \omega_j(0)}{\partial u_k}+\omega_j(0)\omega_k(0)=\frac{\partial \omega_k(0)}{\partial u_j}+\omega_k(0)\omega_j(0),\qquad 1\leq j\neq k \leq n.
\end{gather}
\eqref{15marzo2018-8} is the Frobenius integrability condition of system \eqref{15marzo2018-2}, as it is shown in the proof of Lemma~\ref{24marzo2018-2}. Explicitly writing \eqref{15marzo2018-1}, a simple calculation yields \eqref{15marzo2018-3}. Now, \eqref{26marzo2018-1} is identically satisfied. Substitution of \eqref{15marzo2018-3} into \eqref{15marzo2018-8} shows, by considering the diagonal, that
\begin{gather*}
\frac{\partial \mathcal{D}_j}{\partial u_k}-\frac{\partial \mathcal{D}_k}{\partial u_j}=0, \qquad j\neq k,
\end{gather*}
so that the matrix valued form $\sum\limits_{j=1}^n\mathcal{D}_j(u) \dd u_j$ is closed. Being the domain simply connected, it is also exact. Substitution of \eqref{15marzo2018-3} into \eqref{15marzo2018-2} and direct calculation shows that
\begin{gather*}
\frac{\partial A_{kk}}{\partial u_1}=\cdots=\frac{\partial A_{kk}}{\partial u_n}=0,\qquad \forall\, k=1,\dots,n.\tag*{\qed}
\end{gather*}\renewcommand{\qed}{}
\end{proof}

\begin{proof}[Proof of Proposition \ref{16marzo2018-1}] The assumption that $\omega_j(z,u)$ has at most a pole at $z=\infty$ implies that Lemma~\ref{16marzo2018-2} applies for $u\in \mathbb{D}^\prime$. Hence
\begin{gather}\label{17marzo2018-1}
\omega_j(z,u)=zE_j+[F_1(u),A(u)]+ \mathcal{D}_j(u) \qquad \text{for $u\in \mathbb{D}^\prime$}.
\end{gather}
Now, Lemma \ref{16marzo2018-3} applies on $\mathbb{D}^\prime$. On the other hand, $\omega_j(z,u)$ and $zE_j+[F_1(u),A(u)]$ are holomorphic on $\mathbb{D}\big(u^0\big)$, therefore~\eqref{17marzo2018-1} holds on~$\mathbb{D}\big(u^0\big)$. All the other statements of Proposition~\ref{16marzo2018-1} follow from Lemmas~\ref{16marzo2018-2} and~\ref{16marzo2018-3}.
\end{proof}

\subsection*{Acknowledgements} I am very grateful to the referee for valuable suggestions which improved the paper, especially concerning the importance of reference~\cite{YT} for the characterisation of weak and strong isomo\-nodromic deformations.

\pdfbookmark[1]{References}{ref}
\LastPageEnding

\end{document}